%% file: main.tex
\RequirePackage{fix-cm}
\documentclass[smallextended]{svjour3}       % onecolumn (second format)
\smartqed  % flush right qed marks, e.g. at end of proof
\usepackage{graphicx}

\usepackage {tikz}
\usepackage{pgfkeys}
\usetikzlibrary{shapes,arrows}
\usepackage{pgfplots}
\pgfplotsset{compat=1.13}
\usepackage[utf8]{inputenc}
\usepackage{fullpage,amsmath,color,xcolor,amssymb,enumitem,setspace}
\usepackage[pagebackref]{hyperref}
\usepackage{graphicx,xcolor,tikz}
\usepackage{subfigure}
\usepackage{algorithmic, algorithm}
\usepackage{xspace,color}

\hypersetup{
  colorlinks   = true, %Colours links instead of ugly boxes
  urlcolor     = blue, %Colour for external hyperlinks
  linkcolor    = blue, %Colour of internal links
  citecolor   = red %Colour of citations
}

\usepackage[numbers]{natbib}
\input{defs.tex}

\newcommand{\modif}[1]{{#1}}

\begin{document}

\title{Optimal Complexity and Certification of Bregman First-Order Methods%\thanks{Grants or other notes
%about the article that should go on the front page should be
%placed here. General acknowledgments should be placed at the end of the article.}
}
% \subtitle{Do you have a subtitle?\\ If so, write it here}

%\titlerunning{Short form of title}        % if too long for running head

\author{Radu-Alexandru Dragomir \and Adrien B. Taylor \and Alexandre d'Aspremont\footnote[1]{Last two authors listed in alphabetical order} \and J\'er\^ome Bolte\footnotemark[1]}

%\authorrunning{Short form of author list} % if too long for running head

\institute{Radu-Alexandru Dragomir \at
              Universit\'e Toulouse I Capitole, Toulouse \& 
D.I. \'Ecole Normale Sup\'erieure, Paris, France. radu-alexandru.dragomir@inria.fr 
           \and
           Adrien B. Taylor \at
    INRIA, D.I. \'Ecole Normale Sup\'erieure, Paris, France. adrien.taylor@inria.fr
\and
            Alexandre d'Aspremont \at            
            CNRS \& D.I., UMR 8548, \'Ecole Normale Sup\'erieure, Paris, France. aspremon@ens.fr
            \and
            J\'er\^ome Bolte \at TSE (Universit\'e Toulouse 1 Capitole), Toulouse, France. jbolte@ut-capitole.fr
}

\date{Last revised on September 29, 2020}

% \author{Radu-Alexandru Dragomir}
% \address{Universit\'e Toulouse I Capitole, Toulouse, \vskip 0ex
% D.I. Ecole Normale Sup\'erieure, Paris, France.}
% \email{radu-alexandru.dragomir@inria.fr}

% \author{Adrien B. Taylor}
% \address{INRIA, D.I. Ecole Normale Sup\'erieure, Paris, France}
% \email{adrien.taylor@inria.fr}

% \author{Alexandre d'Aspremont$^*$}
% \address{CNRS \& D.I., UMR 8548,\vskip 0ex
% \'Ecole Normale Sup\'erieure, Paris, France.}
% \email{aspremon@ens.fr}

% \author{J\'er\^ome Bolte$^*$}
% \address{TSE (Universit\'e Toulouse 1 Capitole),\vskip 0ex
% Toulouse, France.}
% \email{jbolte@ut-capitole.fr}
% \date{\today}

\maketitle
\begin{abstract}
We provide a lower bound showing that the $O(1/k)$ convergence rate of the NoLips method (a.k.a. Bregman Gradient or Mirror Descent) is optimal for the class of problems satisfying the relative smoothness assumption.
This assumption appeared in the recent developments around the Bregman Gradient method, where acceleration remained an open issue.

\modif{The main inspiration behind this lower bound stems from an extension of the performance estimation framework of Drori and Teboulle (Mathematical Programming, 2014) to Bregman first-order methods. This technique allows computing worst-case scenarios for NoLips in the context of relatively-smooth minimization. In particular, we used numerically generated worst-case examples as a basis for obtaining the general lower bound.}
\end{abstract}

\let\thempfn\relax% Remove footnote number printing mechanism
%   \footnotetext[1]{$^*$Last two authors listed in alphabetical order.}
  
% \keywords{Bregman distance \and NoLips \and Lower complexity bounds \and Performance Estimation Problems \and Computer-aided analysis \and Mirror descent}
% \PACS{PACS code1 \and PACS code2 \and more}
% \subclass{MSC code1 \and MSC code2 \and more}

% \setcounter{tocdepth}{3} \tableofcontents % toc for draft revising+, comment before submission
\section{Introduction}\label{s:intro}

% \paragraph{General FOM intro}
We consider the constrained minimization problem
\begin{equation}\label{eq:min_problem_bpg}\tag{P}
    \min_{x \in C} f(x)
\end{equation}
where $f$ is a convex continuously differentiable function and $C$ is a nonempty closed convex subset of $\reals^n$. In large-scale settings, first-order methods are particularly popular due to their simplicity and their low cost per iteration.

The (projected) gradient descent (PG) is a classical method for solving \eqref{eq:min_problem_bpg}, and consists in successively minimizing quadratic approximations of $f$, with
\begin{equation}\label{eq:grad_descent}\tag{PG}
x_{k+1} = \argmin_{u \in C} f(x_k) + \la \nabla f(x_k), u - x_k \ra + \frac{1}{2\lambda} \|u - x_k\|^2,
\end{equation}
where $\|\cdot\|$ is the Euclidean norm. Although standard, there is often no good reason for making such approximations, beyond our capability of solving this intermediate optimization problem. In other words, this traditional approximation typically does not reflect neither the geometry of $f$ nor that of $C$. 
% Furthermore, although apparently nice, each iteration of this procedure involves a projection onto the set $C$, which is typically doable, but costly, even for simple sets.
A powerful generalization of PG consists in performing instead a \textit{Bregman gradient step}
\begin{equation}\label{eq:bregman_descent}\tag{BG}
x_{k+1} = \argmin_{u \in C} f(x_k) + \la \nabla f(x_k), u - x_k \ra + \frac{1}{\lambda} D_h(u,x_k),
\end{equation}
where the Euclidean distance has been replaced by the \textit{Bregman distance} $D_h(x,y) := h(x) - h(y) - \la \nabla h(y),x-y \ra$ induced by some strictly convex and continuously differentiable \textit{kernel function} $h$. A well-chosen $h$ allows designing first-order algorithms adapted to the geometry of the constraint set and/or the objective function. Of course, a conflicting goal is to choose $h$ such that each iteration~\eqref{eq:bregman_descent} can be solved efficiently in practice, discarding choices such as $h=f$ (for which performing an iteration would be as hard as solving the original problem).

    Recently, Bauschke et al. \cite{Bauschke2017} introduced a natural condition for analyzing this scheme, assuming that the inner objective in the iteration \eqref{eq:bregman_descent} is an upper bound on $f$. This ensures that performing an iteration decreases the function values $f(x_k)$. This assumption, known as relative smoothness (precisely defined in Def.~\ref{def:h-smooth} below), generalizes the standard $L$-smoothness assumption implied by Lipschitz continuity of $\nabla f$. The Bregman gradient algorithm, also called NoLips in the setting of~\cite{Bauschke2017}, is thus a natural extension of gradient descent~\eqref{eq:grad_descent} to objective functions whose geometry is better modeled by a non-quadratic kernel $h$. Practical examples of relative smoothness arise in Poisson inverse problems \cite{Bauschke2017}, quadratic inverse problems \cite{Bolte2018}, rank minimization \cite{Dragomir2019} and regularized higher-order tensor methods \cite{Nesterov2018}.

\paragraph{Can we accelerate NoLips?} 
In the Euclidean setting where $h(x) = \frac{1}{2}\|\cdot\|^2$, accelerated projected gradient methods exhibit faster convergence than the vanilla projected gradient algorithm. These methods, which can be traced back to Nesterov~\cite{Nesterov1983}, are proven to be \textit{optimal} for $L$-smooth functions and have found a number of successful applications, in e.g., imaging \cite{Beck2009}. A natural question is therefore to understand whether the NoLips algorithm can be accelerated in the relatively-smooth setting. This question has been raised in several works, including that of Bauschke, Bolte and Teboulle~\citep[Section 6]{Bauschke2017}, that of Lu, Freund and Nesterov~\citep[Section 3.4]{Lu2016}, and the survey of Teboulle~\citep[Section 6]{Teboulle2018}. Partial answers have already been provided under somewhat strict additional regularity assumptions (see e.g.,~\cite{Auslender2006,Bayen2015,Hanzely2018} and discussions in the sequel), while the general case was apparently still open, and relevant in practical applications.
In this work, we produce a lower complexity bound proving that NoLips is \textit{optimal} for the general relatively-smooth setting, and therefore that generic acceleration is impossible.

In order to do so, we adopt the standard \emph{black-box model} used for studying complexity of first-order methods \cite{Nemirovski1983}. We consider that both $f$ and $h$ are described by first-order oracles, so as to obtain generic complexity results, and we look for worst-case \textit{couples} of functions $(f,h)$ satisfying the relative smoothness assumption. A central idea in our approach is the fact that, when studying the worst-case behavior of Bregman methods in the relatively-smooth setting, $f$ and $h$ can get arbitrarily close to some \emph{limiting pathological nonsmooth functions}. 
\modif{

\paragraph{Obtaining worst-case scenarios of Bregman first-order methods.} To obtain the lower complexity bound, we start by empirically inspecting the worst-case behaviors of NoLips. In other words, we show that worst-case scenarios (i.e., worst-case pairs of functions ($f,h$)) can be generated numerically through appropriate semidefinite programs (SDP).%, which we solved using appropriate software~\cite{mosek,yalmip}.

The problems of computing such worst-case scenarios are usually referred to as \emph{performance estimation problems} (PEPs), and were pioneered by~\cite{Drori2014} in the context of smooth convex minimization. An additional attractive feature of this approach is that feasible points to their dual problems naturally correspond to worst-case guarantees. For our purposes, we adapt the PEP framework to the setting of Bregman methods and relatively-smooth functions, and showcase the approach by providing worst-case examples for NoLips, along with the corresponding worst-case guarantees coming from its dual. Finally, the very simple and pathological worst-case functions for NoLips served as an inspiration for developing the more general lower bound for Bregman first-order schemes.

Discovering worst-case functions for NoLips is not the only interest of PEPs, as they also allow us to explore worst-case behaviors and convergence bounds for a variety of first-order methods, in a variety of settings, as we illustrate in the sequel.}

\subsection{Contributions and paper organization} 
The main contribution of this work is twofold.
    First, we provide a lower bound showing that it is impossible to generically accelerate Bregman gradient methods under the appropriate oracle model. More precisely, we show that the $O(1/k)$ convergence rate on function values of NoLips is \textit{optimal} in the relatively-smooth setting. As mentioned earlier, the family of worst-case functions that we used for developing the lower bound was inspired by numerical solutions to a series of Performance Estimation Problems (PEPs).
    
    \modif{For this purpose, we developed PEP techniques for Bregman settings. It required extending the analysis of~\cite{Taylor2017} to handle classes of differentiable (but not necessarily $L$-smooth) and strictly convex functions.} While we present the analysis on the basic NoLips algorithm for readability purposes, our results and methodology can be applied to various Bregman methods, such as inertial variants \cite{Auslender2006}, or the Bregman proximal point scheme for convex minimization and monotone inclusions \cite{Eckstein1993,Bui2019}. Besides discovering worst-case examples, PEPs can be used for obtaining bounds with various convergence criteria, as we showcase by proving a new rate on the Bregman divergence between successive iterates for NoLips.

    The paper is organized as follows. After introducing the setup in Section \ref{s:setup}, we prove the optimality of NoLips in Section \ref{s:complexity}. We expose the framework of computer-aided analysis of Bregman methods in Section \ref{s:pep}, including several applications in Section \ref{ss:proofs}. We point out that Sections \ref{s:complexity} and \ref{s:pep} are both of independent interest and can be read separately. 
    
\subsection{Related work}

\paragraph{Bregman methods.} The idea of using non-Euclidean geometries induced by convex kernels can be traced back to the work of Nemirovskii and Yudin \cite{Nemirovski1983}. For nonsmooth objectives, it gave birth to the mirror descent algorithm \cite{Ben-tal2001,Beck2003,Juditsky}, which generalizes the subgradient method to non-Euclidean geometries. It has been proven to be particularly efficient for minimization on the unit simplex, where choosing the \textit{entropy kernel} turns out to be much more effective and scalable than the squared Euclidean norm. This approach has been very successful in online learning; see \cite[Chap. 5]{Bubeck2011} and references therein. The use of Bregman distances has also been thoroughly studied for interior proximal methods \cite{Censor1992,Teboulle1992,Eckstein1993,Auslender2006}.

The introduction of the relative smoothness assumption in \cite{Bauschke2017} has provided a way to adapt the Bregman kernel to the geometry of the objective function $f$ and thus extend the domain of application of the Bregman Gradient method. Subsequent work has focused on nonconvex extensions \cite{Bolte2018}, linear convergence rates under additional assumptions \cite{Lu2016,Bauschke2019}, and inertial variants \cite{Hanzely2018,Mukkamala2019}.

\paragraph{Black-box model and lower complexity bounds.} The first-order black-box model, developed initially in the works of Nemirovskii \cite{Nemirovski1983} and later Nesterov \cite{Nesterov2004} has allowed to prove optimal complexity for several classes of problems in first-order optimization \cite{Drori2017}. 
\modif{These results usually rely on well-chosen \textit{worst-case functions} whose structure makes them difficult to minimize for all methods within a given class. Our worst-case instances are obtained from pointwise maxima of affine functions, reminiscent of lower bounds for nonsmooth convex minimization \cite{Nemirovski1983,Woodworth}. Our construction then involves smoothing those piecewise affine functions, making them differentiable. This technique is also used in the very related work of Guzman and Nemirovskii~\cite{Guzman2015}, which studies lower bounds for minimization of convex functions that are smooth with respect to $\ell_p$ norms. To the best of our knowledge, the lower bound obtained in the sequel is not a particular case of those in~\cite{Guzman2015}, as smoothness with respect to a certain norm is different from relative smoothness with respect to the same (squared) norm, beyond the $\ell_2$-norm.}
% The very related work of Guzman and Nemirovskii~\cite{Guzman2015} studies lower bounds of first-order methods for smooth convex minimization (with a particular focus on smoothness being measured with $l_p$-norms). The smoothing technique we use in the sequel is reminiscent of their technique. To the best of our knowledge, it does not contain the lower bound obtained in the sequel as a particular case.

\paragraph{Performance estimation problems.} The PEP methodology, proposed initially by \cite{Drori2014}, was already used to discover optimal methods and corresponding lower bounds in other settings: for smooth convex minimization~\cite{Drori2014,Kim2016,Drori2017,Drori2019a}, nonsmooth convex minimization~\cite{Drori2016,Drori2019a}, and stochastic optimization~\cite{Drori2019}.

\subsection{Notation}
We use $\overline{C}$ to denote the closure of a set $C$, $\interior C$ for its interior and $\partial C$ for its boundary. We denote $(e_1,\dots,e_n)$ the canonical basis of $\reals^n$, and for $p \in \{1,\dots n\}$ we write $E_p = \text{Span}(e_1,\dots,e_p)$ the set of vectors supported by the first $p$ coordinates. $\symm_n$ denotes the set of symmetric matrices of size $n$. If (P) is an optimization problem, then val(P) stands for its (possibly infinite) value. 

Subscripts on a vector denote the iteration counter, while a superscript such as $x^{(i)}$ denotes the $i$-th coordinate.
The set $I = \{0,1,\dots N,*\}$ is often used to index the first $N$ iterates of an optimization algorithm as well as the optimal point:
\[
\{x_i\}_{i\in I} = \{x_0,x_1,\dots,x_N,x_*\}.
\]
We use the standard notation $\la \cdot,\cdot\ra$ for the Euclidean inner product, and $\|\cdot\|$ for the corresponding Euclidean norm. For a vector $x \in \reals^n$, we write $\|x\|_\infty = \max_{i=1\dots n} |x^{(i)}|$ for its $\ell_\infty$ norm.
% Given a positive semidefinite matrix $A\succeq 0$, we also use the notation $\la \cdot,\cdot \ra_A=\la A \cdot,\cdot \ra$ and denote the corresponding induced semi-norm $\norm{\cdot}_A$. 
% Finally, we use $\Fccp$ to denote the class of closed, proper and convex functions. We also use the following subsets of $\Fccp$: (i) $\FmuL$ the set of $\mu$-strongly convex $L$-smooth functions (with $0\leq \mu\leq L\leq \infty$), (ii) the set $\Fsd$ of strictly convex differentiable functions, (iii) the set $\Fstrict$ of strictly convex functions, and (iv) the set $\Fdiff$ of differentiable convex functions.
Other notations are standard from convex analysis; see e.g., \cite{Rockafellar2008,Bauschke2011}.

\section{Algorithmic setup}\label{s:setup}
In this section, we introduce the base ingredients and technical assumptions on $f$ and $h$ that are used within Bregman first-order methods. %In particular, it is necessary to assume $h$ to be \emph{Legendre} in order to have well-defined iterations of the form~\eqref{eq:bregman_descent}.
\modif{
\subsection{Kernel functions}

Let $C$ be a nonempty closed convex subset of $\reals^n$.
The first step in defining Bregman methods is the choice of a \textit{kernel}  (or reference) function $h$ on $C$. %In particular, when $C=\reals^n$, the technical definition below reduces to requiring $h$ to be continuously differentiable and strictly convex.

\begin{definition}[Kernel function]
A function $h : \reals^n \rightarrow \reals \,\cup \, \{+\infty\} $ is called a kernel function on $C$ if
\begin{enumerate}[label = (\roman*)]
    \item $h$ is closed convex proper (c.c.p.),
    \item $\overline{\dom h} =C$,
    \item $h$ is continuously differentiable and strictly convex on $\interior \dom h \neq \emptyset$.
    % \item $\|\nabla h(x_k)\| \rightarrow \infty$ for every sequence $\{x_k\}_{k \geq 0}\subset \interior \dom h$ converging to a boundary point of $\dom h$ as $k \rightarrow \infty$. 
\end{enumerate}
\end{definition}
}
A kernel function $h$ induces a \textit{Bregman distance} $D_h$ defined as
\begin{equation*}
    D_h(x,y) = h(x) - h(y) - \la \nabla h(y), x-y \ra \quad \forall x \in \dom h, y \in  \dom \nabla h.
\end{equation*}
Note that $D_h$ is not a distance in the classical sense, however it enjoys a separation property; due to the strict convexity of $h$ we have $D_h(x,y) \geq 0  \,\,\, \forall x \in \dom h, y \in \dom \nabla h$, and $D_h(x,y)=0$ iff $x = y$.

\paragraph{Examples.} We list some of the most classical examples of kernel functions:
\begin{itemize}
    \item the \textbf{Euclidean kernel} $h(x) = \frac{1}{2}\|x\|^2$ with domain $\reals^n$, and for which $D_h(x,y) = \frac{1}{2}\|x-y\|^2$ is the Euclidean distance,
    \item the \textbf{Boltzmann-Shannon entropy} $h(x)= \sum_i x^{(i)} \log x^{(i)}$ extended to 0 by setting $0 \log 0= 0$, whose domain is thus $\reals^n_+$,
    \item the \textbf{Burg entropy} $h(x)= \sum_i - \log x^{(i)}$ with domain $\reals^n_{++}$,
    \item the \textbf{quartic kernel} $h(x) = \frac{1}{4}\|x\|^4 + \frac{1}{2}\|x\|^2$ with domain $\reals^n$ \cite{Bolte2018}.
\end{itemize}

We refer the reader to \cite{Bauschke2017,Lu2016} for more examples. It should be emphasized that, while a kernel function is differentiable on the interior of its domain, it is not required to be differentiable on the boundary. For instance, the Boltzmann-Shannon entropy is continuous but not differentiable at~$0$. Moreover, the domain of $h$ can be closed, such as for the Boltzmann-Shannon entropy, or open, as for the Burg entropy.

\paragraph{Convex conjugate.} If $h$ is a kernel function, we define its convex conjugate $h^*$ as
\[h^*(y) = \sup_{u \in \reals^n} \la u,y \ra - h(u)\]
If, for every $y \in \reals^n$, the supremum in the definition of $h^*(y)$ is attained, then $h^*$ is differentiable and its gradient satisfies for every $u \in \dom \nabla h^*$
\[ \nabla h^*(y) = \argmax_{u \in \reals^n}\, \la u,y \ra - h(u).\]
% is also Legendre \cite[Thm 26.5]{Rockafellar2008}, and that its gradient is the inverse of $\nabla h$, that is
% $\nabla h^* = (\nabla h)^{-1}$.

\subsection{Relatively-smooth optimization problems}\label{ss:relsmooth}

% We recall the framework of the NoLips algorithm described in \cite{Bauschke2017} for solving the minimization problem~\eqref{eq:min_problem_bpg}. As we are interested in studying the complexity, we focus here on the simple Bregman gradient method. Our lower bound will be \textit{a fortiori} valid for the Bregman \textit{proximal} gradient algorithm designed for solving composite problems \cite[Eq. (12)]{Bauschke2017}.
We now recall the framework of relatively-smooth optimization \cite{Bauschke2017,Lu2016} for solving the minimization problem 
\begin{equation}\label{eq:min_problem_bpg_recall}\tag{P}
    \min_{x \in C} f(x)
\end{equation}
For simplicity, we present the setting without nonsmooth regularization term; our lower bound is a fortiori valid for the Bregman \textit{proximal} gradient algorithm designed for solving composite problems~\cite[Eq. (12)]{Bauschke2017}.

Let us first state our blanket assumptions.
\modif{
\begin{assumption}\label{assumption_bpg}\
 \begin{enumerate}[label = (\roman*)]
       \item \label{ass:h_leg} $h$ is a kernel function on $C$,
     \item \label{ass:f_diff} $f : \reals^n \rightarrow \reals \cup \{+\infty\}$ is a closed convex proper function such that $\dom h \subset \dom f$ and which is continuously differentiable on $\dom \nabla h$, 
     \item \label{assumption:well_posed} For every $\lambda > 0$, $x \in \interior \dom h$ and $p \in \reals^n$, the problem 
     \[ \min_{u \in C} \la p, u-x \ra + \frac{1}{\lambda} D_h(u,x) \]
     has a unique minimizer, which lies in $\dom \nabla h$,
        \item \label{ass:bounded_below} The problem has at least one minimizer, i.e., $\argmin_C f \neq \emptyset$.
    %  \item \label{ass:argmin_in_dom} There exists at least one minimizer $x_* \in \argmin_C f$ such that $x_* \in \dom h$.
    
        % \item \label{ass:bounded_below} The problem is bounded from below, i.e. $f_* := \inf \,\{f(x) : x\in C\} > -\infty$, and there exists at least one minimizer $x_* \in \argmin_C f$ such that \[x_* \in \dom h.\]
 \end{enumerate}
 
\end{assumption}
}

% \begin{remark}\label{remark:h_smoothness}
Condition \ref{assumption:well_posed} is standard and ensures well-posedness of Bregman gradient methods. It is satisfied if, for instance, $h$ is strongly convex or supercoercive \cite[Lemma 2]{Bauschke2017}. 
% In Condition \ref{ass:argmin_in_dom}, we make the requirement that there is a solution $x_*$ to \eqref{eq:min_problem_bpg} that lies in $\dom h$. This is a nontrivial assumption and we must distinguish two cases:
% \begin{itemize}
%     \item if $\dom h$ is \textbf{closed}, as for the Euclidean kernel and the Boltzmann-Shannon entropy, then $C = \dom h$ and the condition is necessarily satisfied for every minimizer.
%     \item If $\dom h$ is \textbf{open}, like for the Burg entropy, Condition \ref{ass:argmin_in_dom} may fail as the minimizers $x_*$ can lie on the boundary of $\dom h$, where $h$ is infinite.
% \end{itemize}
% \end{remark}
In addition to these assumptions, the central property we need in order to apply the Bregman gradient method is the so-called relative smoothness \cite{Bauschke2017,Lu2016}.

\begin{definition}[Relative smoothness]\label{def:h-smooth}
    Let $h$ be a kernel function on $C$, and $f$ a function such that $\dom h \subset \dom f$. We say that $f$ is smooth relative to $h$ if there exists a constant $L > 0$ such that \begin{equation}\label{eq:rel_smooth}\tag{LC}
        Lh - f \quad \textup{is convex on}\, \dom h.
    \end{equation}
\end{definition}
Relative smoothness allows to build a simple global majorant of $f$; indeed, \eqref{eq:rel_smooth} implies that (see, e.g,~\cite{Bauschke2017})
\begin{equation*}
    f(x) \leq f(y) + \la \nabla f(y), x-y \ra + L D_h(x,y) \quad \forall x \in \dom h, y \in \dom \nabla h,
\end{equation*}
and the NoLips method consists in successively minimizing this upper approximation.

\modif{We list below some examples of relatively-smooth problems.

\begin{itemize}
    \item \textbf{Euclidean case. } When choosing the Euclidean kernel $h(x) = \frac{1}{2}\|x\|^2$, \eqref{eq:rel_smooth} reduces to the usual descent lemma and holds for instance when the gradient of $f$ is Lipschitz continuous. To avoid ambiguity, we refer to this standard Euclidean smoothness as \emph{L-smoothness}.
    \item \textbf{Classical mirror descent setting.} In some previous work on Bregman methods~\cite{Auslender2006,Bayen2015}, it is assumed that $f$ has a Lipschitz continuous gradient with constant $\tilde{L}$ and that the kernel $h$ is $\sigma$-strongly convex. This is a particular case of relative smoothness, since we have
    \begin{equation*}
    \begin{split}
          \left\{
        \begin{array}{ll}
            \nabla f \,\, \text{is Lipschitz continuous with constant } \tilde{L} \\
            h \,\,\text{is}\,\, \sigma-\text{strongly convex}
        \end{array}
    \right. &\implies
    \left\{
        \begin{array}{ll}
            \frac{\tilde{L}}{2} \|\cdot\|^2 - f \,\, \text{is convex} \\
            h - \frac{\sigma}{2}\|\cdot\|^2 \,\,\text{is convex}
        \end{array}
    \right. \\
    &\implies \frac{\tilde{L}}{\sigma}(h - \frac{\sigma}{2} \|\cdot\|^2) + (\frac{\tilde{L}}{2} \|\cdot\|^2 - f) \,\, \text{is convex} \\
    &\implies  \frac{\tilde{L}}{\sigma} h - f \,\, \text{is convex}\\
    &\implies f \text{ is } \text{smooth relative to } h \text{ with constant } \tilde{L}/\sigma.
    \end{split}
    \end{equation*}
    \item \textbf{Poisson inverse problems.} More recent examples include functions that are not $L$-smooth in the Euclidean sense, such as the Kullback-Leibler divergence between some observation $b \in \reals^m$ and a linear measurement $Ax$ of an unknown source vector $x \in \reals^n$:
    \[f(x) = D_{\textup{KL}}(b,Ax) = \sum_{j=1}^m b_j \log \big(\frac{b_j}{A_j x}\big) - A_j x + b_j.\]
    Minimizing $f$ on the nonnegative orthant allows the recovery of a signal corrupted with Poisson noise, which is a fundamental problem in imaging sciences \cite{Review2009}. In this setting, $f$ is not $L$-smooth since its Hessian diverges around the origin. However, it can be shown to be relatively-smooth with respect to the Burg entropy $h(x) = \sum_i -\log(x^{(i)})$ (see \cite{Bauschke2017}).
    \item \textbf{Quartic functions.} A large class of problems in phase recovery and low-rank matrix optimization involve minimizing polynomials of degree $4$. These polynomials are not globally $L$-smooth but are relatively-smooth with respect to the quartic kernel $h(x) = \frac{1}{4}\|x\|^4 + \frac{1}{2}\|x\|^2$ (see \cite{Bolte2018,Dragomir2019}).
\end{itemize}

}
% The relative smoothness assumption generalizes the usual smoothness assumption; in particular, when taking the Euclidean kernel $h(x) = \frac{1}{2}\|x\|^2$, \eqref{eq:rel_smooth} reduces to standard smoothness implied by the Lipschitz continuity of $ \nabla f$. To avoid ambiguity, we will refer to this standard Euclidean smoothness as \emph{L-smoothness}.
% By choosing different kernel functions $h$, it is possible to show that \eqref{eq:rel_smooth} holds for functions that are not $L$-smooth \cite{Bauschke2017,Bolte2018}.

% \paragraph{Remark.} A particular case of $h$-smoothness appears when $f$ has a Lipschitz continuous gradient with constant $\tilde{L}$ and the kernel $h$ is $\sigma$-strongly convex (see e.g.,~\cite{Auslender2006,Bayen2015}), provided that the norm is Euclidean. Indeed, in this case we have 
% \begin{equation*}
% \begin{split}
%       \left\{
%     \begin{array}{ll}
%         \nabla f \,\, \text{is Lipschitz continuous with constant } \tilde{L} \\
%         h \,\,\text{is}\,\, \sigma-\text{strongly convex}
%     \end{array}
% \right. &\implies
% \left\{
%     \begin{array}{ll}
%         \frac{\tilde{L}}{2} \|\cdot\|^2 - f \,\, \text{is convex} \\
%         h - \frac{\sigma}{2}\|\cdot\|^2 \,\,\text{is convex}
%     \end{array}
% \right. \\
% &\implies \frac{\tilde{L}}{\sigma}(h - \frac{\sigma}{2} \|\cdot\|^2) + (\frac{\tilde{L}}{2} \|\cdot\|^2 - f) \,\, \text{is convex} \\
% &\implies  \frac{\tilde{L}}{\sigma} h - f \,\, \text{is convex}
% \end{split}
% \end{equation*}
% which shows that $f$ is $h$-smooth with constant $\tilde{L} /\sigma$.
We use the following convenient notation to characterize the class of relatively-smooth problems.

\begin{definition}
    We say that the couple of functions $(f,h)$ is a relatively-smooth instance, and write~$(f,h) \in \mathcal{B}_L(C)$ if 
     \begin{enumerate}[label = (\roman*)]
        \item $f$ and $h$ satisfy Assumption \ref{assumption_bpg},
        \item $Lh-f$ is convex on $C$.
 \end{enumerate} 
 
    Finally, let us denote by $\mathcal{B}_L$ the union of $\mathcal{B}_L(C)$ for all closed convex sets $C$:
    \[ 
    \mathcal{B}_L = \bigcup_{n \geq 1} \bigcup_{\substack{C \subset \reals^n \\ C \textup{ closed convex}}} \mathcal{B}_L(C) 
    \]
\end{definition}

\modif{
\subsection{The NoLips/Bregman Gradient algorithm}
}

Previous assumptions allow defining the Bregman Gradient (BG)/NoLips algorithm for minimizing~$f$. For simplicity, we only consider the constant step size method.

\begin{algorithm}[H]
{\normalsize
	\begin{algorithmic}
		\STATE \textbf{Input:} $(f,h) \in \mathcal{B}_L(C)$, $x_0 \in \interior \dom h$, step size $\lambda \in (0,1/L]$.
		\FOR{k = 0,1,\dots}
		\STATE 
		    \begin{equation} \label{eq:iteration_nolips}
		    x_{k+1} = \argmin_{u \in \reals^n} \, \la \nabla f(x_k), u-x_k \ra + \frac{1}{\lambda} D_h(u,x_k) \end{equation}
		\ENDFOR
	\end{algorithmic}
	\caption{Bregman Gradient (BG) / NoLips \cite{Bauschke2017}}
	\label{algo:bpg}
}
\end{algorithm}

Using first-order optimality conditions, update~\eqref{eq:iteration_nolips} can alternatively be written as
\begin{equation}
    x_{k+1} = \nabla h^*\left[ \nabla h(x_k) - \lambda \nabla f(x_k) \right]
\end{equation}
 involving the gradient $\nabla h^*$ which is usually referred to as the \textit{mirror map}. \modif{The three operations $\nabla f, \nabla h$ and $\nabla h^*$ are the basic building blocks of Bregman-type methods, which we now define formally.}

 \subsection{Defining a class of Bregman first-order methods}
 
\modif{For proving a general lower bound for relatively-smooth optimization, we need to specify the oracle model and the class of methods under consideration.}
 
We adopt the first-order black-box model, where information about a function can be gained by calling an \textit{oracle} returning the value and gradient of $f$ at a given point. In the Bregman setting, we assume that we also have access to the first-order oracles of the kernel function $h$ and its conjugate $h^*$. 

\modif{\begin{definition}\label{def:algorithm}
An algorithm $\mathcal{A}$ is called a Bregman first-order algorithm if, for a given problem instance $(f,h) \in \mathcal{B}_L$ and number of iterations $T \in \mathbb{N}$, it generates at each time step $t \in \{ 0,\dots, T\}$, a set of primal points $\mathcal{X}_t$ and dual points $\mathcal{Y}_t$ from the following process:
    \begin{enumerate}
        \item Set $\mathcal{X}_0 = \{x_0\}$, where $x_0 \in \interior \dom h$ is some initialization point, and $\mathcal{Y}_0 = { \{\nabla f(x_0),\nabla h(x_0)\}}$.
        \item For each $t =1,\dots T$, perform one of the two following operations:
        \begin{itemize}
            \item either call the \textbf{primal oracle} $(\nabla f, \nabla h)$ at some point $x_t$ chosen such as
            \[ x_t \in \textup{Span}(\mathcal{X}_{t-1}) \cap \dom \nabla h \] and update the dual set as
            \[\mathcal{Y}_{t} = \mathcal{Y}_{t-1} \cup \{\nabla f(x_t), \nabla h(x_t)\}.\]
            \item Or call the \textbf{mirror oracle} $\nabla h^*$ at some dual point $y_t$ chosen such as 
            \[ y_t \in \textup{Span}(\mathcal{Y}_{t-1})\]
            with
            \[\nabla h^*(y_t) = \argmin_{u \in C} h(u) - \la y_t, u\ra\]
            and update the primal set as
            \[\mathcal{X}_{t} = \mathcal{X}_{t-1} \cup \{\nabla  h^*(y_t)\}.\]
        \end{itemize} 
        \item Output some point $x_T \in \textup{Span}(\mathcal{X}_T)$.
    \end{enumerate}
    \end{definition}}
% This model implicitly assumes that $y_t$ is chosen in the domain of the oracle so as to guarantee the existence of the next iterate. 

Such structural assumptions on the class of algorithms are classical from complexity analyses of Euclidean first-order methods and are used to prove e.g., optimality of accelerated first order methods~\cite{Nesterov2004}. Definition~\ref{def:algorithm} is a natural extension to the Bregman setting, allowing additional uses of the oracles associated with the kernel function $h$. This model can often be relaxed through the use of more involved information theoretic arguments, see e.g.,~\cite{Nemirovski1983,Guzman2015,Drori2017,Woodworth}.

Here, we focus on Definition \ref{def:algorithm} as it is general enough to encompass all Bregman-type methods that only use oracles for $\nabla f, \nabla h$, which we call the \textit{primal oracles}, the map $\nabla h^*$, which we call the \textit{mirror oracle}, as well as  linear operations. One can verify that all known Bregman gradient methods, including NoLips and inertial variants such as IGA \cite{Auslender2006} or the recent algorithm in \cite{Hanzely2018}, fit in this model.

Observe that, as NoLips performs one primal oracle call and one mirror call per iteration, an iteration of NoLips corresponds actually to \textit{two time steps} of the formal procedure in Definition \ref{def:algorithm}. This is why, in order to avoid ambiguity, we state our lower bound as a function of the number of oracle calls.

\section{Convergence rate and optimality of NoLips}\label{s:complexity}

In this section, we start by recalling the $O(1/k)$ convergence rate bound for the NoLips algorithm in the setting where $(f,h) \in \mathcal{B}_L(C)$. We then proceed to prove that NoLips is an \textit{optimal} algorithm for the class $\mathcal{B}_L(C)$, by showing that this rate is also a \textit{lower bound} for a generic class of Bregman gradient algorithms that we define below. The key elements for proving the lower bound were empirically discovered through the solution to a Performance Estimation Problem (PEP), which is detailed in Section~\ref{s:pep}.

\subsection{Upper bound}\label{ss:upper_bound}

% We first state the following complexity bound for NoLips, which has also been established in the recent work \citep[Cor. 1]{Zhou2019}. It improves the constant terms of the original bound of Bauschke et al. \cite{Bauschke2017}.
We first state the $O(1/k)$ convergence rate for NoLips. It slightly differs with the one from \cite{Bauschke2017}, as it is improved by a factor of 2 and does not involve the so-called \textit{symmetry coefficient}. 
\begin{theorem}[NoLips convergence rate]\label{thm:nolips_bound}
    Let $L > 0$, $C$ be a nonempty closed convex subset of $\reals^n$ and $(f,h) \in \mathcal{B}_L(C)$ be an relatively-smooth instance. Then the sequence $\{x_k\}_{k \geq 0}$ generated by Algorithm \ref{algo:bpg} with constant step size $\lambda \in (0,1/L]$ satisfies for all $k \geq 0$
    \modif{
    \begin{equation}\label{eq:comp_estimate}
        f(x_k) - f(u) \leq \frac{D_{h}(u,x_0)}{\lambda \, k}
    \end{equation}
    for every $u \in \dom h$. }
\end{theorem}
\modif{
\begin{remark}
Let $x_* \in \argmin_C f$. In order to take $u = x_*$ in Equation \eqref{eq:comp_estimate} and obtain a bound on the suboptimality gap $f(x_k)- f(x_*)$, we need $x_*$ to belong to the domain of $h$. In most cases, this condition is trivially satisfied. However, it can fail if $x_*$ lies on the boundary of $C$ and $\dom h$ is open, such as for the Burg entropy.
\end{remark}
}

The proof of Theorem \ref{thm:nolips_bound}, whose analytical form has been inferred from solving a PEP, is provided in Section \ref{sss:nolips_pep}.
This result extends the $O(1/k)$ rate of Euclidean gradient descent for $L$-smooth functions to the relatively-smooth setting. 
% However, unlike in the Euclidean case, we show in the next section that this rate is actually neither improvable for NoLips, nor for other Bregman first-order methods.

\paragraph{Faster algorithms under additional assumptions.}
It is natural to ask whether an \textit{accelerated} Bregman algorithm can be obtained, with a better convergence rate than $O(1/k)$. This has already been achieved under additional regularity assumptions, as follows
\begin{itemize}
    \item in the Euclidean setting, when $h(x) = \frac{1}{2}\|x\|^2$ and $f$ is $L$-smooth, the seminal accelerated gradient method of Nesterov \cite{Nesterov1983} enjoys a $O(1/k^2)$ convergence rate, which is optimal for this class of functions \cite{Nesterov2004}.
    \item When $h$ is a strongly convex kernel with closed domain and $f$ is $L$-smooth (which, as discussed in Section \ref{ss:relsmooth}, is a particular case of relative smoothness), the Improved Interior Gradient Algorithm~(IGA)~\cite{Auslender2006} also admits a $O(1/k^2)$ convergence rate using the same momentum technique as Nesterov-type methods. 
    \item Recently, \cite{Hanzely2018} proposed an accelerated Bregman proximal gradient algorithm with rate $O(1/k^\gamma)$, where $\gamma \in [1,2]$ is determined by some crucial \textit{triangle scaling property} of the Bregman distance, whose genericity is unclear.
\end{itemize}

However, the existence of an accelerated algorithm for the general relatively-smooth setting was still an open question prior to this work. Indeed, many applications such as Poisson inverse problems \cite{Bauschke2017} or D-optimal design \cite{Lu2016} do not satisfy the supplementary assumptions made in the works mentioned above.
In the next section, we prove that, up to a constant factor of $2$, the bound \eqref{eq:comp_estimate} is not improvable in general for Bregman-type methods, making NoLips an \textit{optimal} algorithm in the black box setting for $(f,h)\in\mathcal{B}_L$.

\subsection{A lower bound for relatively-smooth Bregman optimization} 
\label{ss:lower_bound}

We show in Theorem \ref{thm:low_bound} below that for any $\epsilon \in (0,1)$ and number of oracle calls $N$, there is a pair of functions $(f,h) \in  \mathcal{B}_L(\reals^{2N+1})$ and some $x_0 \in \reals^{2N+1}$ such that for any \textit{Bregman gradient algorithm} initialized at $x_0$, the output $x_N$ returned after performing at most $N$ oracle calls satisfies
\begin{equation}\label{eq:low_bound_first}
            f(x_N) - \min_{\reals^{2N+1}}f \geq (1-\epsilon)\,\frac{LD_{h}(x_0, x_*)}{2N+1}.
\end{equation}

\paragraph{Proof intuition.} For finding an instance $(f,h)$ which is difficult for all Bregman methods, we use two main ideas. The first is the well-known technique used by Nesterov \cite{Nesterov2004} for proving that $O(1/k^2)$ is the optimal complexity for $L$-smooth convex minimization. He defines a ``worst function in the world" that allows any gradient method to discover only one dimension per iteration, hence \textit{hiding} the minimizer from the algorithm in the remaining unexplored dimensions.

The second idea is more specific to our setting, and relies on the fact that the set of relatively-smooth problems $\mathcal{B}_L(C)$ is not closed. In particular, a limit of differentiable functions need not be differentiable. \modif{Thence,  we actually build} a worst-case \textbf{sequence} of differentiable functions parameterized by some parameter $\mu$, whose limit when $\mu \rightarrow 0$ is a nonsmooth pathological function. 

\paragraph{Choosing the objective function.}
Let us fix a dimension $n \geq 1$ and a positive constant $\eta > 0$. Define the convex function $\hat{f}$ for $x \in \reals^n$ by
\begin{equation*}
    \hat{f}(x) = \max_{i=1,\dots,n} |x^{(i)} - 1-\frac{\eta}{i}| = \|x-x_*\|_\infty 
\end{equation*}
which has an optimal value $\hat{f_*} = 0$ attained at
\[x_* := (1+\eta,1+\frac{\eta}{2},\dots,1+\frac{\eta}{n}).\]
The behavior of $\hat{f}$ as a \textit{pathological function} comes from the fact that if at least one of the coordinates of $x$ is zero, then $\hat{f}(x)-\hat{f}_* \geq 1$. Let us first prove a technical lemma about the subdifferential of $\hat{f}$.

\begin{lemma}\label{lemma:subdiff_fhat}
    Let $x \in \reals^n$ and $v \in \partial \hat{f}(x)$ be a subgradient of $\hat{f}$ at $x$. Then
    \begin{enumerate}[label = (\roman*)]
        \item \label{item:first_prop_subgrad} $\|v\|_\infty \leq 1$.
        \item \label{item:second_prop_subgrad} Let $i \in \{1\dots n\}$. If $v^{(i)} \neq 0$ then $|x^{(i)}-x_*^{(i)}| = \|x-x_*\|_\infty$.
    \end{enumerate}
\end{lemma}

\begin{proof}
    Write $\hat{f}$ as  $
        \hat{f}(x) = \max_{1 \leq i \leq n} \hat{f}_i(x)
    $
    with $\hat{f}_i(x) = |x^{(i)} - x^{(i)}_*|$. Then, by \citep[Lemma 3.1.10]{Nesterov2004}, we have \[
        \partial \hat{f}(x) = \text{Conv} \,\{ \partial \hat{f}_i(x) | i \in I(x)\}
    \]
    where $I(x) = \{i \in \{1\dots n\}\,|\, \hat{f}_i(x) = \hat{f}(x)\}$.
    Hence, \ref{item:first_prop_subgrad} follows immediately from the well-known property that the subgradients of the absolute value lie in $[-1,1]$. \ref{item:second_prop_subgrad} is a consequence of the fact that if $v^{(i)} \neq 0$, then $i \in I(x)$, which means that $|x^{(i)}-x_*^{(i)}| = \|x-x_*\|_\infty$.
\qed
\end{proof}

Note that $\hat{f}$ is nonsmooth hence does not meet our assumptions. We approach it with a differentiable function by considering its Moreau envelope $f_\mu$ given by 
\begin{equation}\label{eq:def_fmu}
    f_\mu(x) = \min_{u \in \reals^n} \hat{f}(u) + \frac{1}{2 \mu} \| x-u \|^2
\end{equation}
where $\mu \in (0,1)$ is a small parameter. $f_\mu$ is a smoothed version of $\hat{f}$, which behaves similarly to $\hat{f}$ when we choose $\mu$ small enough. Figure \ref{fig:f_mu} illustrates this phenomenon in two dimensions.

For general properties of the Moreau proximal envelope, we refer to \cite{Moreau1965}. Let us state some properties of $f_\mu$ that we need for the analysis.

\begin{lemma}\label{lemma:fmu}
    $f_\mu$ is a differentiable convex function, whose minimizers are the same as those of $\hat{f}$. Its gradient at a point $x \in \reals^n$ is given by $\nabla f_\mu(x) = \mu^{-1}\left(x- \textup{prox}_{\hat{f}}^\mu\left(x\right)\right)$ where 
    
    \begin{equation*}
        \textup{prox}_{\hat{f}}^\mu(x) = \argmin_{u \in \reals^n} \hat{f}(u) + \frac{1}{2\mu} \|x-u\|^2
    \end{equation*}
    is the Moreau proximal map. Moreover, $\nabla f_\mu$ is Lipschitz continuous with constant $1/\mu$.
\end{lemma}

\begin{figure}
  \centering
  \subfigure[$\hat{f}$]{
    \includegraphics[width = 0.3\columnwidth]{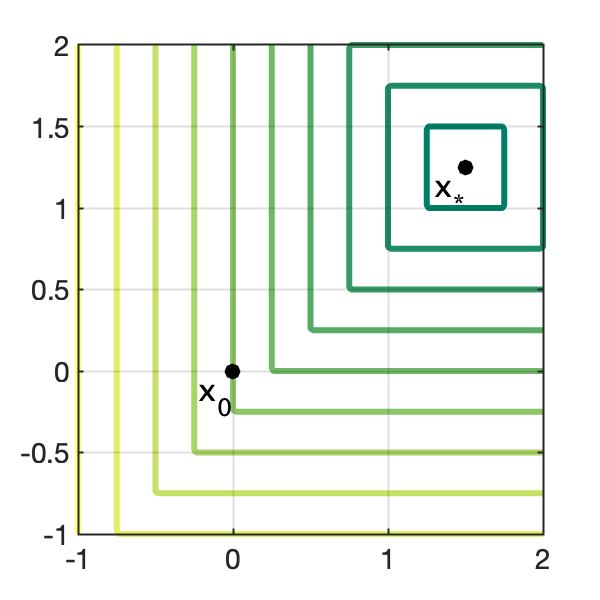}
  }
  \quad
  \subfigure[$f_\mu$]{
    \includegraphics[width = 0.3\columnwidth]{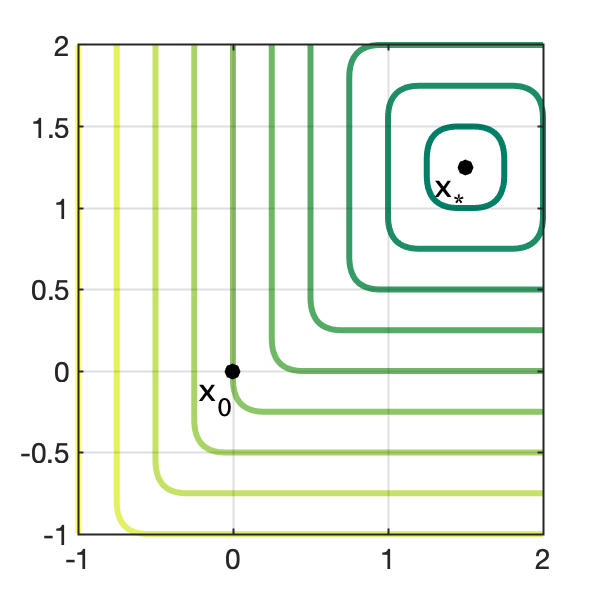}
  }

\caption{ Level curves of function $\hat{f}$ (left) and of its smoothed Moreau evelope $f_\mu$ (right) for $n = 2,\mu = 0.2$ and $\eta = 1/2$. Lemma \ref{lemma:f_zero_pres} states that if $\mu$ is small enough compared to $\eta$, the behaviors of $\hat{f}$ and $f_\mu$ at $x_0 = 0$ are the same. Indeed, the size of the smoothed region where the corners are ``rounded" decreases when $\mu$ goes to 0.}
\label{fig:f_mu}
\end{figure}

Let us now prove the central property of $f_\mu$, which states that when the last $n - p$ coordinates of $x$ are small enough, the gradient $\nabla f_\mu(x)$ is supported on the first $p+1$ coordinates.
Recall that we denote $(e_1,\dots,e_n)$ the canonical basis of $\reals^n$ and write, for $p \in \{1\dots n\}$, $E_p = \text{Span}(e_1,\dots,e_p)$ and $E_0 = \{(0,\dots, 0)\}$.

\begin{lemma}\label{lemma:f_zero_pres}
Assume that $\mu \in (0,1)$ and $\eta > 4 \mu n^2$.
    Let $p \in \{0\dots n-1\}$. For any vector $x \in \reals^n$ such that \[\max_{i=p+1,\dots, n} |x^{(i)}| \leq \mu\]
    we have that $\nabla f_\mu(x) \in E_{p+1}$. In addition, we have $\|\nabla f_\mu(x)\|_\infty \leq 1$.
\end{lemma}

\begin{proof}
Take $x \in \reals^n$ such that $\max_{i=p+1,\dots, n} |x_i| \leq \mu$. By Lemma \ref{lemma:fmu}, $\nabla f_\mu$ is given by
\begin{equation}\label{eq:prox_grad}
    \nabla f_\mu(x) = \frac{1}{\mu}(x-\text{prox}^\mu_{\hat{f}}(x))
\end{equation}
Write $y = \text{prox}^\mu_{\hat{f}}(x)$. The optimality condition defining the proximal map yields
\begin{equation}\label{eq:prox_cond}
    y - x + \mu v = 0
\end{equation}
where $v \in \partial \hat{f}(y)$, and therefore the combination of  \eqref{eq:prox_grad} and \eqref{eq:prox_cond} implies 
\begin{equation}\label{eq:grad_fmu}
    \nabla f_\mu(x) = v \in \partial \hat f(y).
\end{equation}

Now, let us assume by contradiction that $\nabla f_\mu(x) $ is not in $E_{p+1}$, meaning that there exists an index \\$l \in \{p+2\dots n\}$ such that $v^{(l)} \neq 0$.
It follows from Lemma \ref{lemma:subdiff_fhat} that $|(y-x_*)^{(l)}| = \|y-x_*\|_\infty$. Hence we have in particular that $|y^{(l)}-x_*^{(l)}| \geq |y^{(p+1)}-x^{(p+1)}_*|$. Using Condition \eqref{eq:prox_cond} to replace $y$ we get \[
        | x^{(l)}_* + \mu v^{(l)} - x^{(l)}| \geq | x^{(p+1)}_* + \mu v^{(p+1)} - x^{(p+1)}|,
\]
and recalling the definition of $x_*$ we have
\[
        | 1 + \frac{\eta}{l} + \mu v^{(l)} - x^{(l)}| \geq | 1 + \frac{\eta}{p+1} + \mu v^{(p+1)} - x^{(p+1)}|.
\]
By Lemma \ref{lemma:subdiff_fhat}, $\|v\|_\infty\leq1$, so for all $i$ we have $1 + \mu v^{(i)} \geq 1 - \mu \|v\|_\infty \geq 0$. In addition, we assumed that $\max_{i=p+1,\dots, n} |x^{(i)}| \leq \mu < \frac{\eta}{4n^2}$ which implies $\frac{\eta}{i} - x^{(i)} \geq 0$ for all $i \geq p+1$. Therefore, both terms inside the absolute values are nonnegative, it follows that we can drop absolute values and 
\begin{equation}
\begin{split}
          \mu(v^{(l)}-v^{(p+1)})  &\geq   \frac{\eta}{p+1}-\frac{\eta}{l} + x^{(l)} - x^{(p+1)} \\
          &\geq \eta \cdot\frac{l - (p+1)}{l(p+1)} - 2 \mu \\ 
          &\geq \frac{\eta}{l(p+1)} - 2\mu \\
                    &\geq \frac{\eta}{n^2} - 2\mu, \\
\end{split}
\end{equation}
and therefore 
\[
    v^{(l)} - v^{(p+1)} \geq \frac{\eta}{\mu n^2} - 2 > 2,
\]
because we assumed $\eta > 4\mu n^2$. This is a contradiction since $(v^{(l)} - v^{(p+1)}) \leq 2 \|v\|_\infty \leq 2$. Finally, the second part of the lemma is a consequence of \eqref{eq:grad_fmu} and $\|v\|_\infty \leq 1$.
\qed
\end{proof}

We also need the following lemma for relating the values of $\hat{f}$ and $f_\mu$.

\begin{lemma}\label{lemma:value_fmu}
    Let $\mu > 0$ and $x \in \reals^n$. Then
    $f_\mu(x) \geq \hat{f}(x) - \mu$.
\end{lemma}

\begin{proof}
Write $y = \text{prox}^\mu_{\hat{f}}(x)$. By definition of $f_\mu$ and the proximal map we have
\begin{equation*}
\begin{split}
    f_\mu(x) &= \hat{f}(y) + \frac{1}{2\mu} \|y - x\|^2 \\
    &\geq \hat{f}(y) \\
    &= \|y-x_*\|_\infty \\
    &\geq \|x - x_*\|_\infty - \|x - y\|_\infty.
    \end{split}
\end{equation*}
    Recall the optimality conditions defining the proximal map can be written as
    \[\mu^{-1}(x-y) \in \partial f(y),\]
    and, since all subgradients of $\hat{f}$ have coordinates smaller than 1 (Lemma \ref{lemma:subdiff_fhat}), we reach $\|x-y\|_\infty \leq \mu$. It follows that
    $
    f_\mu(x) \geq \|x-x_*\|_\infty - \|x-y\|_\infty \geq \|x-x_*\|_\infty - \mu = \hat{f}(x) - \mu$, which concludes the proof.
    \qed
\end{proof}

\paragraph{Choosing the kernel.} As for the objective function $f_\mu$, let us pick a \text{family} of kernels $h_\mu$, whose behavior approach those of a nonsmooth function as $\mu \rightarrow 0$.

Let us first define a unidimensional convex function $\phi_\mu: \reals \rightarrow \reals$ by

\begin{equation*}
    \phi_\mu(t) = \left\{
    \begin{array}{ll}
        t -\mu/2   & \mbox{ if  } t \geq 
        \mu, \\
        \frac{1}{2\mu} t^2 & \mbox{ elsewhere}.
    \end{array}
\right.
\end{equation*}\modif{Note that $\phi_\mu$ is sometimes known as the \textit{Huber function}, which is a smooth approximation of the absolute value and also appears as a worst-case function for first-order methods in $L$-smooth minimization~\cite{Taylor2017}.} 

Define $d_\mu : \reals^n \rightarrow \reals$ through  
\begin{equation}\label{eq:def_dmu}
    d_\mu(x) = \frac{\mu}{2}\|x\|^2 + \sum_{i=1}^n \phi_\mu(x^{(i)}), \:x \in \reals^n.
\end{equation}
$d_\mu$ is a differentiable strictly convex function, whose gradient satisfies, for $x\in \reals^n$ and $i \in \{1\dots n\}$,
\[\nabla d_\mu(x)^{(i)} = \mu x^{(i)} + \min(1,x^{(i)}/\mu).\]
From the expression above, we can deduce two crucial properties that we need in the sequel: for $x \in \reals^n$ and $i \in \{1\dots n\}$, we have 
\begin{align}
        \label{eq:cond_dmu1}
      &\nabla d_\mu(x)^{(i)} = 0 \quad\text{if and only if}\quad x^{(i)} = 0,\\
      \label{eq:cond_dmu2}
      |&\nabla d_\mu(x)^{(i)}| \leq 1 \quad \,\,\, \text{implies} \quad \quad \, |x^{(i)}| \leq \mu.
\end{align}
Let $L > 0$. We define the kernel $h_\mu$ for $x \in \reals^n$ as
\begin{equation}\label{eq:def_hmu}
    h_\mu(x) = \frac{1}{L}\left(f_\mu(x) + d_\mu(x)\right).
\end{equation}
By construction, $L h_\mu - f_\mu$ is convex, so the relative smoothness property holds. It is easy to see that Assumption~\ref{assumption_bpg} is satisfied as $h_\mu$ is strongly convex, so we have $(f_\mu, h_\mu) \in \mathcal{B}_L(\reals^n)$.

\paragraph{Proving the zero-preserving property of the oracles.} Now that the functions are defined, we are ready to prove that all oracles involved in the Bregman algorithm allow to discover \textit{only one dimension per oracle call.}

\begin{proposition}[Zero-preserving property of $\nabla f_\mu,\nabla h_\mu, \nabla h_\mu^*$]\label{prop:zero_pres}
    Assume that $\mu \in (0,1)$ and $\eta > 4 \mu n^2$.
    Let $p \in\{ 0\dots n-1\}$, and $x \in \reals^n \cap E_p$ a vector supported by the $p$ first coordinates. Then
    \[\nabla f_\mu(x),\nabla h_\mu(x),\nabla h_\mu^*(x) \in E_{p+1}.\]
        %  \begin{enumerate}[label = (\roman*)]
        %     \item $\nabla f_\mu(x) \in E_{p+1}$
        %     \item $\nabla h_\mu(x) \in E_{p+1}$
        %     \item $\nabla h_\mu(x)^* \in E_{p+1}$
        % \end{enumerate}
\end{proposition}

\begin{proof}
    Let $x\in E_p$. Then $x$ satisfies the assumption of Lemma \ref{lemma:f_zero_pres} which proves that $\nabla f_\mu(x) \in E_{p+1}$. By Property \eqref{eq:cond_dmu1} of $d_\mu$, we also have that $\nabla d_\mu(x) \in E_{p}$, which allows us to conclude that 
    \[\nabla h_\mu(x) = L^{-1}\left(\nabla f_\mu(x) + \nabla d_\mu \left(x\right) \right) \in E_{p+1}.\]
    It remains to prove the result for $\nabla h_\mu^*(x)$. Write $z = \nabla h_\mu^*(x)$, which amounts to say that $\nabla h_\mu(z) = x$, that is
    \[
    \nabla f_\mu(z) + \nabla d_\mu(z) = Lx
    \]
    using~\eqref{eq:def_hmu}. Let $l \in \{p+1 \dots n\}$. We have $x \in E_p$, hence the $l-th$ coordinate of $x$ is zero and
    \[
    \nabla f_\mu(z)^{(l)} + \nabla d_\mu(z)^{(l)} = 0.
    \] 
    Using the second part of Lemma \ref{lemma:f_zero_pres}, we have that $\|\nabla f_\mu(z)\|_\infty \leq 1$; it follows that
      \[
      | \nabla d_\mu(z)^{(l)}| \leq 1,
      \]
      which implies that $|z^{(l)}| \leq \mu$, by property \eqref{eq:cond_dmu2} of $d_\mu$. Since this holds for any $l \geq p+1$, we have established
      \[\max_{l=p+1,\dots,n} |z^{(l)}| \leq \mu.\]
      Applying Lemma \ref{lemma:f_zero_pres} to $z$, we obtain that  $\nabla f_\mu(z) \in E_{p+1}$.
    Remembering that $\nabla h_\mu(z) = x \in E_p$ by construction, we get
    \[
    \nabla d_\mu(z) = L \nabla h_\mu(z) - \nabla f_\mu(z) \in E_{p+1}.
    \] 
    By Property \eqref{eq:cond_dmu1} of $d_\mu$, it follows that $z \in E_{p+1}$, which concludes the proof.
    \qed
\end{proof}
 We can now use Proposition \ref{prop:zero_pres} inductively to state a lower bound on the performance of any Bregman gradient algorithm applied to $(f_\mu,h_\mu)$.

\begin{proposition} \label{prop:performance}
Let $N \geq 1$ and choose the dimension $n = 2N+1$.
Let $\mu \in (0,1)$ and $\eta > 4 \mu n^2 $. Consider the functions $f_\mu, h_\mu : \reals^n \rightarrow \reals$ defined in \eqref{eq:def_fmu} and \eqref{eq:def_hmu} respectively. Then, for any Bregman gradient method satisfying Definition \ref{def:algorithm}, applied to $(f_\mu,h_\mu)$ and initialized at $x_0 = (0,\dots 0)$, the output $\overx$ returned after performing at most $N$ calls to each one of the primal and mirror oracles satisfies
% and which performs at most $N$ calls to each one of the primal and mirror oracles, the output $\overline{x}$ satisfies
\[    f_\mu(\overline{x}) - \min_{\reals^n}f_\mu \geq \frac{LD_{h_\mu}(x_*,x_0)}{2N+1} \cdot \frac{1 - \mu}{1 + \mu + \eta + \frac{\mu}{2}(1+ \eta)^2}.
\]
\end{proposition}
\begin{proof} The zero-preserving property and the structure of Bregman gradient algorithms described in Definition \ref{def:algorithm} implies that the set of primal points $\mathcal{X}_t$ and dual points $\mathcal{Y}_t$ at iteration $t$ are supported by the  $t$ first coordinates, i.e., \[
\mathcal{X}_t, \mathcal{Y}_t \subset E_t.
\]
Indeed, since we initialized $\mathcal{X}_0 = \{x_0\} \subset E_0$, this follows by induction.
\modif{Assume that at time $t$, we have $\mathcal{X}_t, \mathcal{Y}_t \subset E_t$. If the primal oracle is chosen at iteration $t+1$, since the query point $x_{t+1}$ is taken as a linear combination of points in $\mathcal{X}_t$ it also lies in $E_t$, and thus Proposition \ref{prop:zero_pres} states that the new dual vectors $\nabla f_\mu(x_{t+1}),\nabla h_\mu(x_{t+1})$ belong to $E_{t+1}$. 
If, on the other hand, the mirror oracle is chosen, then with the same argument we have that $y_{t+1} \in E_t$ an by Proposition \ref{prop:zero_pres} that $\nabla h_\mu^*(y_{t+1}) \in E_{t+1}$}.

Now, because the algorithm has called at most $N$ times each oracle, it has performed at most $2N$ steps and thus the output point satisfies $\overline{x} \in E_{2N}$, which means that $\overline{x}^{(2N+1)} = 0$.

We use Lemma \ref{lemma:value_fmu} to relate $f_\mu(\overline{x})$ and $\hat{f}(\overline{x})$. Recalling that $\min f_\mu= \hat{f}_* = 0$, we get
\begin{equation}\label{eq:f_low_bound}
\begin{split}
    f_\mu(\overline{x}) - \min_{\reals^n}f_\mu &= f_\mu(\overline{x})\\
    &\geq \hat{f}(\overx)- \mu\\
    &\geq  | \overline{x}^{(2N+1)}- x_*^{(2N+1)}| - \mu \\ 
    & = 1 + \frac{\eta}{2N+1} - \mu \\
    &\geq 1 - \mu,
\end{split}
\end{equation}
where we used the definition of $\hat{f}$ and the fact that $\overline{x}^{(2N+1)} = 0$. 

Let us now upper bound the initial diameter. Remembering that $Lh_\mu = f_\mu + d_\mu$ in~\eqref{eq:def_hmu}, we have
\[
LD_{h_\mu}(x_*,x_0) = D_{f_\mu}(x_*,x_0) + D_{d_\mu}(x_*,x_0).
\] 
by definition of the Bregman distance. To deal with the first term, we recall that $f_\mu(x_*) = 0$ and write
\begin{equation*}
    \begin{split}
        D_{f_\mu}(x_*,x_0) &= f_\mu(x_*) - f_\mu(x_0) - \la \nabla f_\mu(x_0),x_*-x_0 \ra\\
        &= - f_\mu(x_0) - \la \nabla f_\mu(x_0),x_*-x_0 \ra\\
        &\leq -\hat{f}(x_0) + \mu - \la \nabla f_\mu(x_0),x_*-x_0 \ra\\
        &= - 1 -\eta + \mu - \la \nabla f_\mu(x_0),x_*-x_0 \ra,
    \end{split}
\end{equation*}
where we used again Lemma \ref{lemma:value_fmu} at $x_0 = (0,\dots,0)$. Now, Lemma \ref{lemma:f_zero_pres} applies to $x_0$ with $p = 0$ and allows to state that $\nabla f_\mu(x_0) \in E_1$ and that $\|\nabla f_\mu(x_0)\|_\infty \leq 1$. Therefore 
\[|\la \nabla f_\mu(x_0),x_*-x_0 \ra| =|\nabla f_\mu(x_0)^{(1)}\,(x_*^{(1)} - x_0^{(1)})|  \leq |x_*^{(1)} - x_0^{(1)}| = 1 + \eta.\]
Hence we have the following upper bound
\begin{equation}\label{eq:bound_dfmu}
\begin{split}
    D_{f_\mu}(x_*,x_0) &\leq  -1 -\eta + \mu + | \la \nabla f_\mu(x_0),x_*-x_0 \ra| \\
    &\leq \mu.
    \end{split}
\end{equation}

The second term can be directly computed from Definition \eqref{eq:def_dmu} of $d_\mu$, recalling that $x_*^{(i)} \geq 1 \geq \mu$ for $i \in \{0\dots n\}$,
\begin{equation}\label{eq:dh_upp_bound}
\begin{split}
        D_{d_\mu}(x_*,x_0)&= d_\mu(x_*) - d_\mu(x_0) - \la \nabla d_\mu(x_0), x_* - x_0 \ra \\
        &= d_\mu(x_*)\\
        &= \sum_{k=1}^{2N+1} \left[\frac{\mu}{2}(1+\frac{\eta}{k})^2 + 1 + \frac{\eta}{k} - \frac{\mu}{2}\right] \\
        & \leq (2N+1) \left[ \frac{\mu}{2}(1+ \eta)^2 + \eta + 1  \right].
\end{split}    
\end{equation}
Combining \eqref{eq:bound_dfmu} and \eqref{eq:dh_upp_bound} gives
\begin{equation*}
    \begin{split}
        LD_{h_\mu}(x_*,x_0) &= D_{f_\mu}(x_*,x_0) + D_{d_\mu}(x_*,x_0)\\ 
        &\leq \mu + (2N+1) \left[ \frac{\mu}{2}(1+ \eta)^2 + \eta + 1  \right]\\
        &\leq (2N+1) \left[\mu + \frac{\mu}{2}(1+ \eta)^2 + \eta + 1  \right].
    \end{split}    
\end{equation*}
This bound, along with \eqref{eq:f_low_bound}, yields
\begin{equation*}
    f_\mu(\overline{x}) - \min_{\reals^n}f_\mu\geq 1 - \mu \geq \frac{L D_{h_\mu}(x_*,x_0)}{2N+1} \cdot \frac{1 - \mu}{1 + \mu + \eta + \frac{\mu}{2}(1+ \eta)^2}
\end{equation*}
whence the desired result.
\qed
\end{proof}

Since constants $\mu,\eta$ can be taken arbitrarily small, we now use Proposition \ref{prop:zero_pres} to show that the bound can be approached to any precision and thus prove our main result.

\begin{theorem}[Lower complexity bound for $\mathcal{B}_L$]\label{thm:low_bound}
Let $N\geq 1$, a precision $\epsilon \in (0,1)$ and let ${x_0 \in \reals^{2N+1}}$ be a starting point. Then, there exist functions $(f,h) \in \mathcal{B}_L(\reals^{2N+1})$ such that for any Bregman gradient method $\mathcal{A}$ satisfying Definition \ref{def:algorithm} and initialized at $x_0$, the output $\overx$ returned after performing at most $N$ calls to each one of the primal and mirror oracles satisfies
\begin{equation*}
    f(\overline{x}) - \min_{\reals^{2N+1}}f \geq \frac{LD_{h}(x_*, x_0)}{2N+1} \cdot(1-\epsilon).
\end{equation*}
\end{theorem}

\begin{proof}
Consider a number $N$ of oracle calls and a target precision $\epsilon \in (0,1)$.
Choose the functions $f_\mu, h_\mu$ defined respectively in Equations \eqref{eq:def_fmu} and \eqref{eq:def_hmu} on $\reals^n$ with $n = 2N+1$. These functions satisfy Assumption~\ref{assumption_bpg}, since their domain is $\reals^n$, they are convex, differentiable, and $h_\mu$ is strongly convex. Moreover, relative smoothness holds because $Lh_\mu - f_\mu = d_\mu$ is convex by construction. Hence $(f_\mu,h_\mu) \in \mathcal{B}_L(\reals^{n})$.

Because the class of problems $\mathcal{B}_L(\reals^n)$ is invariant by translation, we can assume without loss of generality that the algorithm is initialized at $x_0 = (0,\dots 0)$.
Recall that the only conditions our analysis imposed on the parameters $\eta,\mu$ are that $\mu \in (0,1)$ and $\eta > 4 \mu n^2$.

Let us then choose $\eta = \epsilon/4$ and $\mu = \eta/(5n^2) = \epsilon / (20n^2)$. Under these conditions, Proposition \ref{prop:performance} applies and gives that for any point $\overline{x}$ returned by a Bregman gradient algorithm that is initialized at $x_0$ and which performs at most $N$ calls to each oracle we have

\[
    f_\mu(\overline{x}) - \min_{\reals^{2N+1}}f_\mu \geq \frac{LD_{h_\mu}(x_*, x_0)}{2N+1} \cdot \frac{1 - \mu}{1 + \mu + \eta + \frac{\mu}{2}(1+\eta)^2}.
\]
The last term can be bounded from below, using our choice of $\mu,\eta$, and the fact that $\eta < 1$, as 
% \[
% \begin{split}
%     \frac{1 - \mu}{1 + \eta + \mu + \frac{\mu}{2}(1+\eta)^2} &\geq \frac{1 - \mu}{1 + \eta + 3\mu }\\
%     &= \frac{1-\frac{\epsilon}{20n^2}}{1  + \frac{\epsilon}{4} + \frac{3\epsilon}{20n^2}}\\
%     &\geq \frac{1-\epsilon/2}{1+\epsilon/2}\\
%     & \geq 1 - \epsilon
% \end{split}
% \]
\[
    \frac{1 - \mu}{1 + \eta + \mu + \frac{\mu}{2}(1+\eta)^2} \geq \frac{1 - \mu}{1 + \eta + 3\mu }= \frac{1-\frac{\epsilon}{20n^2}}{1  + \frac{\epsilon}{4} + \frac{3\epsilon}{20n^2}} \geq 1 - \epsilon
\]
yielding the desired result.
\qed
\end{proof}

% Idem, with the alternative assumption we can say
\begin{remark}
One can refine the result above in the case where the primal and mirror oracles are not used the same number of times. Indeed, if the primal oracles are called $N_1$ times and the mirror oracle is called $N_2$ times, then the same reasoning shows that the lower bound remains true by replacing $2N$ with $N_1+N_2$.
% \begin{remark}

Our lower bound involves the relative smoothness constant $L$ instead of the step size $\lambda$ in~\eqref{eq:comp_estimate}, but it is equivalent (up to a factor 2) when choosing $\lambda = 1/L$, which is actually the best possible step size choice. \modif{This shows the optimality of NoLips within  the class of Bregman first-order methods (up to a universal constant).}
\end{remark}

% \subsection{Discussion}
% related with Guzman-Nemirovskii

% l_1/l inf setup and conditional gradients
\textit{Connection with Conditional Gradient and the $\ell_\infty$ setting.} 
\modif{
The worst-case function used for the lower bound involves the smoothing of an $\ell_\infty$ norm. As pointed out by one of the referees, there might be a connection between the hardness of the relatively-smooth setting and the lower bound for smooth minimization on the $\ell_\infty$ ball as done in Guzman and Nemirovskii \cite{Guzman2015}. This lower bound, which is also $O(1/k)$, is used by the authors to prove that the rate of the Conditional Gradient algorithm is near-optimal in this setting.

It might be insightful to examine connections between these settings in future works, for example by exploiting duality between Bregman gradient methods and Conditional Gradient, as in~\cite{Bach2015}.
}

\section{Computer-aided performance analyses of Bregman first-order methods}\label{s:pep}

    In this section, we extend the computer-aided performance estimation framework in~\cite{Drori2014,Taylor2015} to the setting of Bregman methods. In short, these results show how to compute the worst-case convergence rate of a given algorithm by solving a numerical optimization problem, called performance estimation problem~(PEP). Solving a PEP offers several benefits, including:
\begin{enumerate}
    \item Computing (numerically) the \textit{exact} worst-case complexity of an algorithm on a given class of problems after a fixed number of iterations.
    \item Studying the corresponding worst-case functions.
    \item Inferring an analytical worst-case guarantee by obtaining a feasible point to the dual PEP. Such dual feasible points correspond to combinations of inequalities that certify the convergence bound.
\end{enumerate}

Here, we focus on inferring worst-case functions. \modif{We used this methodology for guessing, and then designing, the lower bound provided in Section~\ref{ss:lower_bound}}. However, solving  PEPs is also useful for proving new convergence rates (see Section \ref{sss:other_bound}), or for getting quick numerical insights into the convergence properties of an algorithm, like for instance on the inertial algorithm IGA \cite{Auslender2006} (Section \ref{ss:iga}).

To use PEPs on Bregman methods, we extend the analysis in~\cite{Drori2014,Taylor2015} to deal with differentiable and/or strictly convex functions. Previous works on the topic modelled differentiability through an $L$-smoothness condition, and strict convexity through strong convexity, which are assumptions that we avoid in the Bregman setting. The key difference in our work is that the classes of differentiable and/or strictly convex functions are \emph{open} sets. Thus, the worst-case functions for this class might lie on the closure of this set and exhibit some pathological nonsmooth behavior.

This section is organized as follows. In Section \ref{ss:pep_intro}, we introduce the PEP framework. Sections \ref{ss:interpolation}-\ref{s:tightpeps} extend PEPs to the Bregman setting. We provide in Section \ref{ss:proofs} several applications, including the procedure used to find the worst-case functions involved in the proof of the general lower bound in Section \ref{ss:lower_bound}.

\subsection{Worst-case scenarios through optimization}\label{ss:pep_intro}
We now formulate the task of finding the worst-case performance of Algorithm \ref{algo:bpg} as an optimization problem. We focus on the analysis of NoLips for simplicity. However, the same ideas are directly applicable to other Bregman-type algorithms like IGA~\cite{Auslender2006} (see Section \ref{ss:iga}) or Bregman proximal point~\cite{Eckstein1993}. 

Recall that we write $\mathcal{B}_L(C)$ for the set of function pairs $(f,h)$ satisfying Assumption~\ref{assumption_bpg}, such that $Lh - f$ is convex on a convex set $C$. For simplicity, we first focus on the case when  functions have full domain, i.e., $C=\reals^n$ for some $n \geq 1$. In this setting, the set $\mathcal{B}_L(\reals^n)$ can be rewritten as 
\begin{equation*}
    \mathcal{B}_L(\reals^n)=  \left\{\begin{array}{l@{\hskip 1em}|@{\hskip 1em}l}
        & f \mbox{ is convex, differentiable and has at least one minimizer,}\\
        &h \mbox{ is strictly convex and differentiable,}\\
        f,h : \reals^n \rightarrow \reals  & Lh - f \mbox{ is convex,}\\
        &\forall \lambda > 0, \,\forall x,p\in \reals^n,\mbox{ the function } u \mapsto \la p,u-x \ra + \frac{1}{\lambda} D_h(u,x)\\
        &\quad \mbox{ has a unique minimizer.}\\
        \end{array}\right\}, 
\end{equation*}
since all constraints in Assumption~\ref{assumption_bpg} about the domains of $f$ and $h$ become irrelevant. The general case when $C$ is a convex subset of $\reals^n$ can be treated along the same approach. In fact, from the perspective of performance estimation, we can show that every problem in $\bl(C)$ can be reduced to some problem in $\bl(\reals^n)$ with equivalent convergence rate (see Appendix~\ref{a:domh} for details).
% The ingredients are exactly the same as the ones introduced below with the ones that are already known for the Euclidean settings (see e.g.,~\cite{Drori2014,Taylor2015,Drori2016,Taylor2017}
% [+``Contributions to the complexity analysis of optimization algorithms''] 
% \comAT{introduire BL ou bien dire d'une autre fa\c{c}on qu'on regarde le cas $C=\mathbb{R}^n$ dans cette section}

\paragraph{Performance estimation problem.} Throughout this section, we fix a number of iterations $N \geq 1$, a relative smoothness parameter $L > 0$, and a step size $\lambda > 0$. In the currently known analyses of NoLips, worst-case guarantees have the following form 
\begin{equation}\label{eq:generic_bound}
    f(x_N)-f(x_*)\leq \theta(N,L,\lambda) D_h(x_*,x_0).\end{equation}
For instance, Theorem \ref{thm:nolips_bound} states this result with $\theta(N,L,\lambda) = 1/(\lambda N)$ when $\lambda \in (0,1/L]$ (note that since we consider the case where $C = \reals^n$, we can take $x_*$ in the bound as $x_* \in \dom h$ trivially). We then naturally seek the smallest $\theta(N,L,\lambda)$ such that the bound \eqref{eq:generic_bound} holds for any couple $(f,h) \in \mathcal{B}_L(\reals^n)$, that is, solve the optimization problem
% \begin{equation}\label{eq:pep_nolips}\tag{PEP}
% \begin{split}
%       \sup_{f,h,x_0,\dots,x_N,x_*,n} & \big(f\left(x_N\right) - f \left(x_*\right)\big) / D_{h}(x_*,x_0)\\%\frac{f(x_N) - f(x_*)}{D_{h}(x_*,x_0)}\\
%         \mbox{subject to   } &(f,h) \in \mathcal{B}_L(\reals^n), \\
%         % & Lh - f \mbox{ is convex,} \\
%         & x_* \mbox{ is a minimizer of $f$}, \\
%         & x_1,\dots,x_N \mbox{ are generated from $x_0$ by Algorithm \ref{algo:bpg} with step size $\lambda$,}
% \end{split}
% \end{equation}
\begin{equation}\label{eq:pep_nolips}\tag{PEP}
\BA{ll}
       \mbox{maximize} & \big(f\left(x_N\right) - f \left(x_*\right)\big) / D_{h}(x_*,x_0)\\
       \rule{0pt}{12pt}
        \mbox{subject to   } &(f,h) \in \mathcal{B}_L(\reals^n), \\
        % & Lh - f \mbox{ is convex,} \\
        & x_* \mbox{ is a minimizer of $f$}, \\
        & x_1,\dots,x_N \mbox{ are generated from $x_0$ by Algorithm \ref{algo:bpg} with step size $\lambda$,}\\
\EA
\end{equation}
in the variables $f,h,x_0,\dots,x_N,x_*,n$. We refer to this problem as a performance estimation problem (PEP). We use the convention $0/0=0$, so that the objective is well defined when $x_* = x_0$. Optimizing over the dimension $n$ to get dimension-free bounds allows for the problem to admit efficient convex reformulations, as we see in the sequel. We look for guarantees that are independent of the kernel $h$, and therefore $h$ is also an optimization variable.

Let us start by simplifying the problem. First, due to the strict convexity of $h$, the NoLips iteration \eqref{eq:iteration_nolips} can be equivalently formulated via the first-order optimality condition
\[ \nabla h(x_{i+1}) = \nabla h(x_i) - \lambda \nabla f(x_i) \quad \forall i \in \{0\dots N-1\}\]
and, since the domain is $\reals^n$, the condition that $x_*$ minimizes $f$ reduces to requiring $\nabla f(x_*) = 0$. Second, the problem is homogeneous in $(f,h)$ (i.e., from a feasible couple $(f,h)$, take any constant $c>0$ and observe that the couple $(cf,ch)$ is also feasible with the same objective value), hence optimizing the objective function $f(x_N)-f(x_*)$ under the additional constraint $D_h(x_*,x_0)=1$ produces the same optimal value as the problem above.

Finally, we use the same argument as in \cite{Drori2014,Taylor2017} and observe that the objective of \eqref{eq:pep_nolips} and the algorithmic constraints mentioned above depend solely on the values of the first-order oracles of $f$ and $h$ at the points $x_0,\dots,x_N,x_*$. Denoting $I = \{0,1,\dots,N,*\}$ the indices associated with the points involved in the problem, we proceed to write these values as
    \begin{equation*}
        \begin{split}
            \{(f_i, g_i)\}_{i \in I} &= \left\{\big(f(x_i), \nabla f(x_i)\big)\right\}_{i \in I}, \\
            \{(h_i, s_i)\}_{i \in I} &= \{\big(h(x_i), \nabla h(x_i)\big)\}_{i \in I}.
        \end{split}
    \end{equation*}
Using those elements, the iterations of NoLips can be expressed as $s_{i+1} = s_i - \lambda g_i$ for $i \in \{0\dots N-1\}$, and the normalization constraint $D_h(x_*,x_0) = 1$ becomes $h_* - h_0  - \la s_0, x_* - x_0 \ra = 1$.

Using those \textit{discrete} representations of $f$ and $h$, we can reformulate \eqref{eq:pep_nolips} as
\begin{equation*}\label{eq:pep_disc}\tag{PEP}
\BA{rcl}
        %\text{val}\eqref{eq:pep_nolips} =
        &  \mbox{maximize} & f_N - f_* \\
        % &\quad \{f_i, g_i, h_i, s_i\}_{i \in I} = \{f(x_i), \nabla f(x_i), h(x_i), \nabla h(x_i)\}_{i \in I},\\
        \rule{0pt}{11pt}
        &\mbox{subject to} &  f_i = f(x_i), g_i = \nabla f(x_i),\\
        && h_i = h(x_i), s_i = \nabla h(x_i),\quad  \mbox{for all $ i \in I$ and some } (f,h) \in \mathcal{B}_L(\reals^n),\\
        &            & g_* = 0,\\
        &            & s_{i+1} = s_i - \lambda g_i \quad \mbox{for } i\in \{1\dots N-1\},\\
        &            &  h_* - h_0  - \la s_0, x_* - x_0 \ra = 1,
\EA
\end{equation*}
in the variables $n,\{(x_i,f_i,g_i,h_i,s_i)\}_{i \in I}$. 
% \begin{equation*}\label{eq:pep_disc}%\tag{PEP-finite}
% \setlength{\jot}{1pt}
% \begin{split}
%       \text{val}\eqref{eq:pep_nolips} = \sup_{n,\{x_i,f_i,g_i,h_i,s_i\}_{i\in I}} & f_N - f_*\\
%         \mbox{subject to: } & \mbox{there exist } (f,h) \in \mathcal{B}_L(\reals^n) \mbox{ for which }\\
%         % &\quad \{f_i, g_i, h_i, s_i\}_{i \in I} = \{f(x_i), \nabla f(x_i), h(x_i), \nabla h(x_i)\}_{i \in I},\\
%         & \quad f_i = f(x_i), g_i = \nabla f(x_i),\\
%         & \quad h_i = h(x_i), s_i = \nabla h(x_i),\\
%                     & g_* = 0,\\
%                     & s_{i+1} = s_i - \lambda g_i \quad \mbox{for } i\in \{1\dots N-1\},\\
%                     &  h_* - h_0  - \la s_0, x_* - x_0 \ra = 1.
%      \end{split}
% \end{equation*}
The equivalence with the initial problem is guaranteed by the first two constraints which are called the \textit{interpolation conditions}.

It turns out that interpolation conditions for the class $\mathcal{B}_L(\reals^n)$ are delicate to establish, due to assumptions on $h$. Fortunately, there exist two classes $\blstrict(\reals^n)$ and $\blnonsmooth(\reals^n)$ for which they can be derived. The first class is a restriction of $\mathcal{B}_L(\reals^n)$ where $f$ and $Lh-f$ are both assumed to be strictly convex:
\[\blstrict(\reals^n) = \mathcal{B}_L(\reals^n) \cap \{(f,h):\reals^n \rightarrow \reals \,|\, f \mbox{ and } Lh-f \mbox{ are strictly convex} \}\]
whereas the second class consists in considering a relaxation with possibly nonsmooth functions:
\[ 
\blnonsmooth(\reals^n) = \{ (f,h) : \reals^n \rightarrow \reals \,|\, f \mbox{ and } Lh-f \mbox{ are convex}\}.
\]
The following inclusions then directly hold \[
\blstrict(\reals^n) \subset \mathcal{B}_L(\reals^n) \subset \blnonsmooth(\reals^n).
\]
With theses classes, we can now define two easier problems. The first one is a restriction of \eqref{eq:pep_nolips} defined on the class $\blstrict(\reals^n)$, under the additional constraint that all iterates are distinct:
% \begin{equation}\label{eq:pep_underestimation}\tag{\underline{PEP}}
% \setlength{\jot}{1pt}
% \begin{split}
%       \sup_{n,\{(x_i,f_i,g_i,h_i,s_i)\}_{i\in I}} & f_N - f_*\\
%         \mbox{subject to: } & \mbox{there exist } (f,h) \in \blstrict(\reals^n) \mbox{ for which }\\
%         % &\quad \{f_i, g_i, h_i, s_i\}_{i \in I} = \{f(x_i), \nabla f(x_i), h(x_i), \nabla h(x_i)\}_{i \in I},\\
%         & \quad f_i = f(x_i),\, g_i = \nabla f(x_i),\\
%         & \quad h_i = h(x_i),\, s_i = \nabla h(x_i),\\
%         & g_* = 0,\\
%         & s_{i+1} = s_i - \lambda g_i \quad \mbox{for } i\in \{1\dots N-1\},\\
%         &  h_* - h_0  - \la s_0, x_* - x_0 \ra = 1,\\
%         & x_i \neq x_j \quad \mbox{for } i \neq j.\\
%      \end{split}
% \end{equation}
\begin{equation}\label{eq:pep_underestimation}\tag{$\underline{\text{PEP}}$}
\BA{rcl}
        &\mbox{maximize} & f_N - f_* \\
        \rule{0pt}{11pt}
        &\mbox{subject to} &  f_i = f(x_i), g_i = \nabla f(x_i),\\
        && h_i = h(x_i), s_i = \nabla h(x_i),\quad  \mbox{for all $ i \in I$ and some } (f,h) \in \blstrict(\reals^n),\\
        &            & g_* = 0,\\
        &            & s_{i+1} = s_i - \lambda g_i \quad \mbox{for } i\in \{1\dots N-1\},\\
        &            &  h_* - h_0  - \la s_0, x_* - x_0 \ra = 1,\\
        &            & x_i \neq x_j \quad \mbox{for } i \neq j \in I,\\
\EA
\end{equation}
in the variables $n,\{(x_i,f_i,g_i,h_i,s_i)\}_{i \in I}$. 
The second problem is a relaxation of \eqref{eq:pep_nolips}, where $(f,h)\in \blnonsmooth(\reals^n)$ are possibly nonsmooth and $g_i,s_i$ are thus \textit{subgradients}:
% \begin{equation}\label{eq:pep_overestimation}\tag{$\overline{\text{PEP}}$}
% \setlength{\jot}{1pt}
% \begin{split}
%       \sup_{n,\{(x_i,f_i,g_i,h_i,s_i)\}_{i\in I}} & f_N - f_*\\
%         \mbox{subject to: } & \mbox{there exist } (f,h) \in \blnonsmooth(\reals^n) \mbox{ for which }\\
%         % &\quad \{f_i, g_i, h_i, s_i\}_{i \in I} = \{f(x_i), \nabla f(x_i), h(x_i), \nabla h(x_i)\}_{i \in I},\\
%       & \quad f_i = f(x_i),\, g_i \in \partial f(x_i),\\
%         & \quad h_i = h(x_i),\, s_i \in \partial h(x_i),\\
%         & \quad  Ls_i - g_i \in \partial (Lh-f)(x_i),\\
%                     & g_* = 0,\\
%                     & s_{i+1} = s_i - \lambda g_i \quad \mbox{for } i\in \{1\dots N-1\},\\
%                     &  h_* - h_0  - \la s_0, x_* - x_0 \ra = 1.
%      \end{split}
% \end{equation}
\begin{equation}\label{eq:pep_overestimation}\tag{$\overline{\text{PEP}}$}
\BA{rcl}
        &\mbox{maximize} & f_N - f_* \\
        \rule{0pt}{11pt}
        &\mbox{subject to} &  f_i = f(x_i), g_i \in \partial f(x_i),\\
        && h_i = h(x_i), s_i \in \partial h(x_i),\\
        && Ls_i - g_i \in \partial(Lh-f)(x_i) \quad  \mbox{for all $ i \in I$ and some } (f,h) \in \blnonsmooth(\reals^n),\\
        &            & g_* = 0,\\
        &            & s_{i+1} = s_i - \lambda g_i \quad \mbox{for } i\in \{1\dots N-1\},\\
        &            &  h_* - h_0  - \la s_0, x_* - x_0 \ra = 1,\\
\EA
\end{equation}
in the variables $n,\{(x_i,f_i,g_i,h_i,s_i)\}_{i \in I}$. 
We added the technical constraint $ Ls_i - g_i \in \partial (Lh-f)(x_i)$, which is redundant for differentiable functions; but that is necessary in order to establish interpolation conditions in the nonsmooth case. 

Because of the inclusions between the feasible sets of these problems, we naturally have
\[\val{\eqref{eq:pep_underestimation}} \leq \val{\eqref{eq:pep_nolips}} \leq \val{\eqref{eq:pep_overestimation}}.\] 
We prove in the sequel that \eqref{eq:pep_overestimation} can be solved via a semidefinite program and that $\val\eqref{eq:pep_underestimation} = \val\eqref{eq:pep_overestimation}$ (Theorem \ref{thm:equiv}), allowing to reach our claims.

% \begin{remark}
Note that the relaxed problem \eqref{eq:pep_overestimation} does not correspond to any practical algorithm, as NoLips is not properly defined for nonsmooth functions $h$. However, we see in the sequel that feasible points of this problem correspond to accumulation points of \eqref{eq:pep_nolips}. In other words, instances of NoLips can get arbitrarily close to pathological nonsmooth functions whose behaviors are captured by \eqref{eq:pep_overestimation}.
% Indeed, by imposing $s_{i+1}=s_i-\lambda g_i$, we force the fact that the algorithm uses the previously obtained subgradient $s_i$ in order to proceed. Alternatively, evaluating another subgradient $s_i'\in\partial h(x_i)$ and iterating $s_{i+1}=s_i'-\lambda g_i$ would certainly result in a weaker method. In other words, our analysis reduces to a particular way of applying NoLips to the nonsmooth case.
% \end{remark}
% \begin{remark}
% Other types of bounds, such as on the gradient norm $\normsq{\nabla f(x_N)}$ can be obtained simply by changing the objective function, e.g., to $\normsq{\nabla f(x_N)}/{D_h(x_*,x_0)}$ and proceed with the same steps. This point is largely explained in~\cite{Taylor2015} and we therefore do not further discuss this topic.
% \end{remark}

In the following sections, we show that problems \eqref{eq:pep_underestimation} and \eqref{eq:pep_overestimation} can be cast as semidefinite programs (SDP) \cite{Vandenberghe96} and solved numerically using standard packages \cite{mosek,yalmip}. The main ingredient consists in showing that interpolation constraints can actually be expressed using quadratic inequalities, as detailed in the next section.

\subsection{Interpolation involving differentiability and strict convexity}\label{ss:interpolation}

In this section, we show how to reformulate interpolation constraints for~\eqref{eq:pep_underestimation} and~\eqref{eq:pep_overestimation} as quadratic inequalities. We start by recalling interpolation conditions for the class of $L$-smooth and $\mu$-strongly convex functions.

\begin{theorem}[Smooth strongly convex interpolation, \cite{Taylor2017}]
    \label{thm:interp_mul}
Let $I$ be a finite index set, $\{(x_i,f_i,g_i)\}_{i \in I} \in (\reals^n \times \reals \times \reals^n)^{|I|}$ and $0 \leq \mu \leq L \leq +\infty$. The following statements are equivalent:
     
         \begin{enumerate}[label = (\roman*)]
        \item There exists a proper closed convex function $f:\reals^n \rightarrow \reals \cup \{+\infty\}$ such that $f$ is $\mu$-strongly convex, has a $L$-Lipschitz continuous gradient and
        \[f_i = f(x_i),\, g_i \in \partial f(x_i) \quad \forall i \in I.\]
        \item For every $i,j \in I$ we have 
        \begin{equation*}
            \begin{split}
                f_i - f_j - \la g_j, x_i - x_j \ra &\geq \frac{1}{2L}  \|g_i - g_j\|^2 + \frac{\mu }{2 (1 - \mu/L)} \|x_i - x_j - \frac{1}{L}(g_i - g_j)\|^2.
            \end{split}
        \end{equation*}
    \end{enumerate}

\end{theorem}
In particular, when $L=+\infty$ (meaning that we require no smoothness) and $\mu=0$, those conditions reduce to the simpler \emph{convex interpolation} conditions, reminiscent of subgradient inequalities:
\begin{equation}\label{e:subgradient}
f_i - f_j - \la g_j, x_i-x_j \ra \geq 0
\end{equation}
In our setting, we want to avoid working with smoothness and strong convexity, so we provide interpolation conditions for the class of differentiable strictly convex functions.

\begin{proposition}[Differentiable and strictly convex interpolation]\label{prop:interp_cond}
    Let $I$ be a finite index set and $\{(x_i,f_i,g_i)\}_{i \in I} \in (\reals^n \times \reals \times \reals^n)^{|I|}$. The following statements are equivalent:
    
    \begin{enumerate}[label = (\roman*)]
        \item There exists a convex function $f:\reals^n \rightarrow \reals$ such that $f$ is differentiable, strictly convex and
        \[f_i = f(x_i),\, g_i = \nabla f(x_i)\quad \forall i \in I.\]
        \item For every $i,j\in I$ we have 
        % \begin{equation}\label{eq:interp_cond}
        %     \begin{cases}
        %         f_i - f_j - \la g_j, x_i - x_j \ra &\geq 0,\\
        %         f_i - f_j - \la g_j, x_i - x_j \ra  & >0 \mbox{  if   } x_i \neq x_j  \mbox{ (strict convexity),}\\
        %         f_i - f_j - \la g_j, x_i - x_j \ra  & >0 \mbox{ if   } g_i \neq g_j \mbox{ (differentiability).}
        %     \end{cases}
        % \end{equation}
        \modif{
        \begin{equation}\label{eq:interp_cond}
        \begin{aligned}
        \left\{\begin{array}{ll}
        f_i- f_j-\langle g_j,x_i-x_j\rangle >0   &  \text{if } x_i\neq x_j,\\
        f_i=f_j \text{ and } g_i=g_j     & \text{otherwise.}
        \end{array}\right.
        \end{aligned}
        \end{equation}
        }
    \end{enumerate}
\end{proposition}

\begin{proof}
    \modif{(i) $\implies$ (ii). Assume that (i) holds, and choose such a function $f$. The first inequality of \eqref{eq:interp_cond} follows from strict convexity of $f$, and the second line is a consequence of the fact that a differentiable convex function has a unique subgradient at each point \citep[Thm 25.1]{Rockafellar2008}.}
    
%     The second inequality follows directly from strict convexity when $x_i \neq x_j$. 
%     Now, to prove the third part, consider the case when we have $\nabla f(x_i) \neq \nabla f(x_j)$ for some indices $i,j$. Let us prove the result by contradiction, i.e., assume that
%     \begin{equation}\label{eq:contra_assump}
%     f(x_i) - f(x_j) - \la \nabla f(x_j),x_i - x_j \ra = 0.
% \end{equation}
% Let $u \in \reals^n$, convexity implies that
% \[
% f(u) \geq f(x_j) + \la \nabla f(x_j), u-x_j \ra.
% \]
% Combining the above inequality with \eqref{eq:contra_assump} gives
% \[
% f(u) \geq f(x_i) + \la \nabla f(x_j), u-x_i\ra \quad \forall u\in \reals^n
% \]    
%     which shows, by definition of a subgradient, that $\nabla f(x_j) \in \partial f(x_i)$. Since $f$ is differentiable at $x_i$, we have by \citep[Thm 25.1]{Rockafellar2008} that $\partial f(x_i) = \{\nabla f(x_i)\}$  which is a contradiction as we assumed $\nabla f(x_i) \neq \nabla f(x_j)$. Thus the third part of \eqref{eq:interp_cond} is proved.
    \modif{
    (ii) $\implies$ (i). Assume (ii). If for all $i,j\in I$, we have $g_i = g_j$ and $x_i = x_j$, then there is only one point and one subgradient to be interpolated, and the result follows immediatly from considering a well-chosen definite quadratic function.
    In the other case, define 
    \[
    \nu = \min_{\substack{i,j \in I \\  x_i \neq x_j}} f_i - f_j - \la g_j, x_i - x_j  \ra.
    \]
    Because of \eqref{eq:interp_cond} and the finiteness of $I$, we have that $\nu > 0$. Now, define $r$ as
    \begin{equation*}
        r = \max_{i,j \in I} \|g_i - g_j\|^2 + \|x_i - x_j\|^2
    \end{equation*}
    so that $r > 0$. Condition \eqref{eq:interp_cond} implies that for all $i,j \in I$ we have
    \begin{equation}\label{eq:min_df_gj}
            f_i - f_j - \la g_j, x_i - x_j \ra \geq \frac{\nu}{r} \left(\|g_i - g_j\|^2 + \|x_i-x_j\|^2 \right).\\
    \end{equation}
    Indeed, if $x_i \neq x_j$, this follows from the definition of $\nu$ and $r$. If $x_i = x_j$ both sides of the inequality are 0 because of the second line in \eqref{eq:interp_cond}.}
    Let us choose two constants $0<\mu<L<+\infty$ such that
    \[
    % \begin{split}
    \frac{1}{L- \mu} \leq \frac{\nu}{r}, \quad
    \frac{\mu}{1 - \mu / L} \leq \frac{\nu}{r},     
    % \end{split}
    \]
     which is possible as it suffices to take $L$ large enough and $\mu$ small enough. %one can verify that this holds for instance if 
    % \[L = \frac{2r}{\nu}, \, \mu = \frac{L}{10+10 \left(\frac{r}{\nu}\right)^2}\]
    We now proceed to show that the interpolation conditions of Theorem \ref{thm:interp_mul} hold with the constants $\mu,L$ defined above.
    Using the inequality $\|u-v\|^2 \leq 2\|u\|^2 + 2\|v\|^2$ and \eqref{eq:min_df_gj}, we get that for all $i,j$,
    \begin{equation*}
        \begin{split}
        &\frac{1}{2L} \|g_i - g_j\|^2 + \frac{\mu }{2(1 - \mu/L)} \|x_i - x_j - \frac{1}{L}(g_i - g_j)\|^2\\
             & \leq \left( \frac{1}{2L}+ \frac{\mu}{L(L - \mu)} \right) \|g_i - g_j\|^2
            + \frac{\mu}{1 - \mu/L} \|x_i-x_j\|^2\\
             & \leq \left( \frac{1}{L}+ \frac{\mu}{L(L - \mu)} \right) \|g_i - g_j\|^2
            + \frac{\mu}{1 - \mu/L} \|x_i-x_j\|^2\\
            &=  \frac{1}{L-\mu} \|g_i - g_j\|^2
            + \frac{\mu }{1 - \mu/L} \|x_i-x_j\|^2\\
            & \leq \frac{\nu}{r} \|g_i - g_j\|^2 + \frac{\nu}{r} \|x_i - x_j\|^2\\
            & \leq f_i - f_j - \la g_j, x_i - x_j \ra.
        \end{split}
    \end{equation*}
Under those conditions, Theorem~\ref{thm:interp_mul} states that there exists a convex function $f$ that interpolates $\{(x_i,f_i,g_i)\}_{i\in I}$  which is $\mu$-strongly convex and has $L$-Lipschitz continuous gradients. A fortiori, since $\mu >0$ and $L < \infty$, $f$ is differentiable and strictly convex.
Finally, $f$ is finite on $\reals^n$ since it is $L$-smooth.
\qed
\end{proof}
% \textit{Remark.}
    % It is easy to adapt the result of Proposition \ref{prop:interp_cond} for only one of the two conditions (strict convexity or differentiability), which amounts to choose only the corresponding inequalities in~\eqref{eq:interp_cond}.
Using these results, we can now formulate interpolation conditions for the problems \eqref{eq:pep_overestimation} and \eqref{eq:pep_underestimation} involving the classes $\blnonsmooth(\reals^n)$ and $\blstrict(\reals^n)$ that were defined in Section \ref{ss:pep_intro}.

\begin{corollary}[Interpolation conditions for \eqref{eq:pep_overestimation}]\label{cor:blnonsmooth}
    Let $I$ be a finite index set and $\{(x_i,f_i,g_i,h_i,s_i)\}_{i \in I} \in (\reals^n \times \reals \times \reals^n \times \reals \times \reals^n)^{|I|}$. The following statements are equivalent.
             \begin{enumerate}[label = (\roman*)]
             \item There exist functions $(f,h) \in \blnonsmooth(\reals^n)$ such that
             \[
             \begin{split}
            f_i = f(x_i),\, g_i \in \partial f(x_i),\\
            h_i = h(x_i),\, s_i \in \partial h(x_i),\\
             Ls_i - g_i \in \partial (Lh-f)(x_i).\\
            \end{split}\]
             \item For all $i,j \in I$ such that $i \neq j$, we have
                  \begin{equation}\label{eq:interp_blnonsmooth}
                     \begin{split}
                         f_i - f_j - \la g_j, x_i- x_j \ra \geq 0, \\
                         (Lh_i-f_i) - (Lh_j-f_j) - \la Ls_j - g_j, x_i- x_j \ra \geq 0.\\
                     \end{split}
                 \end{equation}
            \end{enumerate}
\end{corollary}

\begin{proof} (i) $\implies$(ii) follows immediately from the definition of a subgradient applied to convex functions $f$ and $Lh-f$.\\
Assume that (ii) holds. By the specialization of \eqref{e:subgradient} in  Theorem \ref{thm:interp_mul}, conditions (ii) imply that there exist two convex functions $f,d : \reals^n \rightarrow \reals$ such that 
\[
    \begin{split}
        f_i &= f(x_i), \quad\quad\quad g_i \in \partial f(x_i),\\
        Lh_i - f_i &= d(x_i), \,\, L s_i - g_i \in \partial d(x_i).
    \end{split}
 \]
Defining the convex function $h = (f + d)/L$, we have that $d = Lh -f $, hence $Ls_i - g_i \in \partial(Lh-f)(x_i)$. We also get \[
 \begin{split}
        &h_i = h(x_i), \, s_i \in \partial h(x_i),
 \end{split}
\]
where we used the fact that $L s_i \in \partial f(x_i) + \partial d(x_i) \subset \partial (f+d)(x_i) = L \partial h(x_i)$ (see \cite[Thm 23.8]{Rockafellar2008} for the subdifferential of a sum of convex functions). Hence (i) holds.
\qed
\end{proof}

\begin{corollary}[Interpolation conditions for \eqref{eq:pep_underestimation}]\label{cor:blstrict}
    Let $I$ be a finite index set and $\{(x_i,f_i,g_i,h_i,s_i)\}_{i \in I} \in (\reals^n \times \reals \times \reals^n \times \reals \times \reals^n)^{|I|}$. Assume that $x_i \neq x_j$ for every $i \neq j \in I$. The following statements are equivalent.
             \begin{enumerate}[label = (\roman*)]
             \item There exist functions $(f,h) \in \blstrict(\reals^n)$ such that \[
             \begin{split}
             \quad f_i = f(x_i),\, g_i = \nabla f(x_i),\\
            \quad h_i = h(x_i),\, s_i = \nabla h(x_i).
            \end{split}\]
             \item For all $i,j \in I$ such that $i \neq j$ we have
                  \begin{equation}\label{eq:interp_blstrict}
                     \begin{split}
                         f_i - f_j - \la g_j, x_i- x_j \ra &> 0, \\
                         (Lh_i-f_i) - (Lh_j-f_j) - \la Ls_j - g_j, x_i- x_j \ra &> 0.\\
                     \end{split}
                 \end{equation}
            \end{enumerate}
\end{corollary}

\begin{proof} Note that since we assumed $x_i \neq x_j$ for every $i \neq j$, interpolation conditions of Proposition \ref{prop:interp_cond} reduce to requiring a strict inequality in~\eqref{eq:interp_cond} for every $i\neq j$.
    As before, define $d := Lh-f$. Then since $(f,h) \in \blstrict(\reals^n)$ the functions $f$ and $d$ are differentiable strictly convex, hence (i) $\implies$ (ii) follows simply from strict convexity of these functions.
    
    Conversely, assume (ii). By using Proposition \ref{prop:interp_cond} again, we can interpolate differentiable strictly convex functions $f$ and $d$ and recover $h$ with $h = (f+d)/L$, thus we have naturally $Lh-f$ convex. The function $h$ is thus also differentiable and strictly convex. Moreover, it can be seen from the proof of Proposition~\ref{prop:interp_cond} that the interpolating functions can actually be chosen strongly convex, hence with this choice the well-posedness condition Assumption \ref{assumption_bpg}\ref{assumption:well_posed} holds, and we can conclude that $(f,h)\in \blstrict(\reals^n)$.
\qed
\end{proof}

\subsection{Semidefinite reformulations}\label{ss:sdp}

Now that we established the interpolation conditions for~\eqref{eq:pep_underestimation} and~\eqref{eq:pep_overestimation}, we may use them to obtain semidefinite performance estimation formulations as in \cite{Drori2014,Taylor2017}. This is made possible by observing that interpolation conditions \eqref{eq:interp_blnonsmooth}-\eqref{eq:interp_blstrict} are quadratic inequalities in the problem variables.
  
\newcommand{\sizeP}{3(N+2)}  
  
Let $\{(x_i,f_i,g_i,h_i,s_i)\}_{i\in I}$ be a feasible point of one of the PEPs in dimension $n$. 
% Information about the vectors involved in the problem can be stacked in a matrix $P \in \reals^{n\times (\sizeP)}$  as follows
% \[P = [x_0,\dots, x_N,x_*,\,g_0, \dots g_N,g_*,\,s_0,\dots,s_N,s_*].\]
% Define now the Gram matrix that contains all dot products between $x_i,g_i,s_i$ for $i \in I$:
% \[G = P^T P \in \symm_{\sizeP}\]
We write $G \in \symm_{\sizeP}$ the Gram matrix that contains all dot products between $x_i,g_i,s_i$ for $i \in I$, with
\[
G = 
\BPM
G^{xx} & G^{gx} & G^{sx}\\
G^{gx\top} & G^{gg} & G^{gs}\\
G^{sx\top} & G^{gs\top} & G^{ss}\\
\EPM
\succeq 0
\]
whose size is independent of the dimension $n$, where the blocks are defined as 
\[
G^{xx}_{ij} = \la x_i, x_j \ra,~G^{gx}_{ij} = \la g_i, x_j \ra,~G^{gs}_{ij} = \la g_i, s_j \ra,~G^{gg}_{ij} = \la g_i, g_j \ra, G^{sx}_{ij} = \la s_i, x_j \ra,~G^{ss}_{ij} = \la s_i, s_j \ra,\quad i,j\in I.
\]
Denote by  \[
F = (f_0,\dots,f_N,f_*) \in \reals^{N+2},\quad
H = (h_0,\dots,h_N,h_*) \in \reals^{N+2},
\]
the vectors representing the function values of $f,h$ at the iterates. Finally observe that all the constraints of \eqref{eq:pep_underestimation} and \eqref{eq:pep_overestimation} can be expressed using only $G$, $F$ and $H$. 
% Define the \textit{encoding vectors} $\bx_i,\bg_i,\bs_i \in \reals^{2N+3}$ for $i \in \{0\dots N\}$ by
% \begin{equation}\label{eq:encoding_vectors}
%     \begin{split}
%         \bx_i &= e_i, \\
%         \bg_i &= e_{N+1+i}, \\
%         \bs_i &= e_{2N+3} - \sum_{k=0}^{i-1} \lambda e_{N+1+k},
%     \end{split}
% \end{equation}
% and 
% \[\bx_* = (0,\dots,0),\quad \bg_* = (0,\dots,0),\quad \bs_* = (0,\dots,0),\]
%  so that we have the following convenient relations for every $i \in I$ 
% \begin{equation*}
%     \begin{split}
%         x_i = P\bx_i,  \\
%         g_i = P\bg_i,  \\
%         s_i = P\bs_i, 
%     \end{split}
% \end{equation*}
%  which can be verified from the definition of $P$ and \eqref{eq:lin_rel_sk}. Using these vectors we can express any dot product between the vectors involved in our problem with the Gram matrix $G$; indeed, we have for instance that \[\la x_i, x_j \ra = \la P \bx_i, P \bx_j\ra = \la G \,\bx_i, \bx_j\ra = \la \bx_i,\bx_j \ra_G.\] 
 
 For instance, interpolation conditions~\eqref{eq:interp_blnonsmooth} for $\blnonsmooth(\reals^n)$ rewrite for all $i,j \in I$ as
 \begin{equation*}
     \begin{split}
         f_i - f_j - G^{gx}_{ji} + G^{gx}_{jj} &\geq 0,\\
         (Lh_i-f_i) - (Lh_j-f_j) -  L(G^{sx}_{ji} - G^{sx}_{jj}) + G^{gx}_{ji} - G^{gx}_{jj} &\geq 0.\\
     \end{split}
 \end{equation*}
%  This allows to reformulate the relaxation~\eqref{eq:pep_overestimation} as a semidefinite program in the variables $G,F,H$:
%  \begin{equation}\label{eq:pep_over_sdp}\tag{sdp-$\overline{\text{PEP}}$}
%  \setlength{\jot}{1pt}
%     \begin{split}
%           \sup_{\substack{G \in \,\mathbf{S}_{\sizeP}, F, H \in \reals^{N+2} }} & f_N - f_*\\
%             \textup{subject to: } 
%             & \hspace{14.5em} f_i - f_j - G^{gx}_{ij} + G^{gx}_{jj} \geq 0 \quad \textup{ for } i,j \in I,\\
%             & (Lh_i-f_i) - (Lh_j-f_j) -  L(G^{sx}_{ij} - G^{sx}_{jj}) + G^{gx}_{ij} - G^{gx}_{jj} \geq 0 \quad \mbox{ for } i,j \in I,\\
%         & G^{gg}_{**} = 0,\\
%         & G^{sx}_{i+1,j} = G^{sx}_{ij}-\lambda G^{gx}_{ij}  \quad \mbox{ for } i \in \{0\dots N-1\}, j \in I,\\
%          &    h_* - h_0  - G^{sx}_{0*} + G^{sx}_{00} =1, \\
%         &G \succeq 0.\\
%          \end{split}
%     \end{equation}
This allows us to reformulate the relaxation~\eqref{eq:pep_overestimation} as a semidefinite program, written
\renewcommand{\arraystretch}{1.1}
\begin{equation}\label{eq:pep_over_sdp}\tag{sdp-$\overline{\text{PEP}}$}
\BA{ll}
\mbox{maximize} & f_N - f_*\\
\rule{0pt}{12pt}
\mbox{subject to} &  f_i - f_j - G^{gx}_{ji} + G^{gx}_{jj} \geq 0, \\
         &(Lh_i-f_i) - (Lh_j-f_j) -  L(G^{sx}_{ji} - G^{sx}_{jj}) + G^{gx}_{ji} - G^{gx}_{jj} \geq 0 \quad \mbox { for } i,j \in I,\\
         & G^{gg}_{**} = 0,\\
          & G^{sx}_{i+1,j} = G^{sx}_{ij}-\lambda G^{gx}_{ij}  \quad \mbox{ for } i \in \{0\dots N-1\}, j \in I,\\
         &  h_* - h_0  - G^{sx}_{0*} + G^{sx}_{00} =1, \\
        & G \succeq 0,
\EA
\end{equation}
in the variables $G\in\symm_{\sizeP}$ and $F, H \in \reals^{N+2}$.

Any feasible point of~\eqref{eq:pep_overestimation} can be cast into an admissible point of~\eqref{eq:pep_over_sdp} by computing the semidefinite Gram matrix $G$. Conversely, if $G,F,H$ is an admissible point of \eqref{eq:pep_over_sdp}, then the vectors $\{(x_i,g_i,s_i)\}_{i\in I}$ can be recovered by performing, for instance, a Cholesky decomposition of $G$. Note that we expressed the algorithmic constraint $s_{i+1} = s_i - \lambda g_i$ only through scalar products with the $x_i$'s in the SDP, since only the projection of the gradients on $\text{Span}(\{x_i\}_{i \in I})$ is relevant in the PEPs.
Because interpolation conditions from Corollary \ref{cor:blnonsmooth} are necessary and sufficient, we conclude that the problems are equivalent, that is
     \[\text{val}\eqref{eq:pep_over_sdp}=\text{val}\eqref{eq:pep_overestimation}.\]
 
The rank of $G$ determines the dimension of the interpolated problem. If we look instead for a solution that has a given dimension $n$, this would mean imposing a nonconvex rank constraint on $G$. Our formulation, on the other hand, is convex and finds the best convergence bound that is dimension-independent, which is a usual requirement for \emph{large-scale settings}. \modif{In our setting, given the size of $G$ and the algorithmic constraints, we can show that there exists worst-case instances of dimensions at most $2N+5$. For NoLips, we show in the sequel that it is even possible to find simple worst-cases in a single dimension.}
 
In the same way, the value of  \eqref{eq:pep_underestimation} can be computed as 
 \begin{equation}\label{eq:pep_under_sdp}\tag{$\text{sdp-}\underline{\text{PEP}}$}
\BA{ll}
\mbox{maximize} & f_N - f_*\\
\rule{0pt}{12pt}
\mbox{subject to} & f_i - f_j - G^{gx}_{ji} + G^{gx}_{jj} > 0,\\
         &(Lh_i-f_i) - (Lh_j-f_j) -  L(G^{sx}_{ji} - G^{sx}_{jj}) + G^{gx}_{ji} - G^{gx}_{jj} > 0 \quad \mbox { for } i \neq j \in I,\\
         & G^{gg}_{**} = 0,\\
          & G^{sx}_{i+1,j} = G^{sx}_{ij}-\lambda G^{gx}_{ij}  \quad \mbox{ for } i \in \{0\dots N-1\}, j \in I,\\
         &  h_* - h_0  - G^{sx}_{0*} + G^{sx}_{00} =1, \\
     & G^{xx}_{ii} + G^{xx}_{jj} - 2 G^{xx}_{ij} > 0 \quad \mbox{for } i\neq j \in I,\\ 
        & G \succeq 0,
\EA
\end{equation}
in the variables $G\in\symm_{\sizeP}$ and $F, H \in \reals^{N+2}$, where we used interpolation conditions for $\blstrict(\reals^n)$ from Corollary \ref{cor:blstrict}, since all points $\{x_i\}_{i \in I}$ are constrained to be distinct. Therefore, as above we infer that 
     \[\text{val}\eqref{eq:pep_under_sdp}=\text{val}\eqref{eq:pep_underestimation}.\]
Recalling the hierarchy between the problems, we thus have
\[\text{val}\eqref{eq:pep_under_sdp} \leq \text{val}\eqref{eq:pep_nolips} \leq \text{val}\eqref{eq:pep_over_sdp}.\]
By comparing the two semidefinite programs stated above, one can notice that the only difference is that \eqref{eq:pep_under_sdp} imposes some inequalities of \eqref{eq:pep_over_sdp} to be strict.
In the next section, we use topological arguments to prove that the values of the two problems are actually equal. In fact, strict inequalities have little meaning in numerical optimization (the value of \eqref{eq:pep_under_sdp} is actually a supremum and not a maximum); in our experiments, we focus on \eqref{eq:pep_over_sdp} as solvers usually admit only closed feasible sets.
 \renewcommand{\arraystretch}{1}
% % \setlength{\unitlength}{1mm}
% % \thicklines
% \begin{figure}
%     \centering
%      \begin{tikzpicture}[scale = 0.9]
%         \draw[dashed] (0,0) ellipse (3.5cm and 0.8cm);
%         \draw[dashed] (0,0.4) ellipse (4.3cm and 1.4cm);
%         \draw (0,0.7) ellipse (5cm and 2cm);
%         \node at (0,0) {\eqref{eq:pep_underestimation} $\equiv$ \eqref{eq:pep_under_sdp}};
%         \node at (0,1.2) {\eqref{eq:pep_nolips}};
%         \node at (0,2.2) {\eqref{eq:pep_overestimation} $\equiv$ \eqref{eq:pep_over_sdp}};
%     \end{tikzpicture}
%     \caption{Inclusions between the feasible sets of the different performance estimation problems. The symbol $\equiv$ denotes equivalence up to a reparametrization. In Section \ref{s:tightpeps}, we show that the closure of the feasible set of \eqref{eq:pep_under_sdp} is the feasible set of \eqref{eq:pep_over_sdp}, proving thus that all these problems have the same value.}
%     \label{fig:pep_hierarchy}
% \end{figure}

%\subsection{Tightness of the approach: the limiting nonsmooth pathological behaviors}\label{s:tightpeps}
\subsection{Tightness of the approach: nonsmooth limit behaviors}\label{s:tightpeps}

We are now ready to prove the main result of this section.
\begin{theorem}\label{thm:equiv}
    The value of the performance estimation problem \eqref{eq:pep_nolips} for NoLips is equal to the value of the nonsmooth relaxation \eqref{eq:pep_overestimation}, which can be computed by solving  the semidefinite program \eqref{eq:pep_over_sdp}.
\end{theorem}
\begin{proof} We show that the closure of the feasible set of \eqref{eq:pep_under_sdp} is the feasible set of \eqref{eq:pep_over_sdp}. We first need to prove that the strengthened problem \eqref{eq:pep_underestimation} is feasible, by finding an instance of NoLips where $f$ and $Lh-f $ are strictly convex and such that all iterates are distinct. It suffices for instance to consider two one-dimensional quadratic functions. Define $f,h : \reals \rightarrow \reals$ with
         \[f(x) = \frac{\alpha}{2}x^2, \, h(x) = \frac{1}{2} x^2 \quad \mbox{ where } \alpha = \min \left(\frac{1}{2 \lambda}, \frac{L}{2} \right).\]
         Then $f$ is strictly convex and so is $Lh-f = \frac{L-\alpha}{2}x^2$ since $L-\alpha \geq \frac{L}{2}>0$. The optimum is $x_* = 0$. Choose 
         \[x_0 = \sqrt{2}\]
        for which we have $D_{h}(x_*,x_0) = x_0^2/2 = 1$. Then, Algorithm \ref{algo:bpg} is equivalent to gradient descent and the iterates satisfy
        \[x_N = (1 - \lambda \alpha)^Nx_0.\]
        Since $\alpha \lambda \leq 1/2 < 1$, all the iterates are distinct and therefore we constructed a feasible point of \eqref{eq:pep_underestimation}. 
        Let us therefore write $(G,F,H)$ a corresponding feasible point of \eqref{eq:pep_under_sdp}, and $(\overline{G},\overline{F},\overline{H})$ a feasible point of \eqref{eq:pep_over_sdp}.
        Define the sequence $\{(G^k,F^k,H^k)\}_{k \geq 1}$ as
    \begin{equation*}
        \begin{split}
            G^k &= \frac{1}{k} G + (1 - \frac{1}{k}) \overline{G},\\
            F^k &= \frac{1}{k} F + (1 - \frac{1}{k}) \overline{F},\\
            H^k &= \frac{1}{k} H  + (1 - \frac{1}{k}) \overline{H}.\\
        \end{split}
    \end{equation*}
    
    Then, for every $k \geq 1$, $(G^k,F^k,H^k)$ is still a feasible point of \eqref{eq:pep_under_sdp}, because of convexity of the constraints and the fact that adding a strict inequality to a weak inequality gives a strict inequality. Moreover, the sequence converges to the point $(\overline{G},\overline{f},\overline{h})$ when $k \rightarrow +\infty$.

    Hence we proved that for any feasible point of \eqref{eq:pep_over_sdp}, there is a sequence of admissible points of \eqref{eq:pep_under_sdp} that converge to it.
    Since the objective is linear in the vector $F$ therefore continuous, we deduce that the two problems have the same value:
    \[\text{val}\eqref{eq:pep_under_sdp} = \text{val}\eqref{eq:pep_over_sdp},\]
    which means that $\text{val}\eqref{eq:pep_underestimation} = \text{val}\eqref{eq:pep_overestimation}$. As val\eqref{eq:pep_nolips} lies in between these two values, we conclude that they are all equal.
    \qed
\end{proof}

Theorem~\ref{thm:equiv} states that the value of the original problem \eqref{eq:pep_nolips} can be computed numerically with a semidefinite solver applied to \eqref{eq:pep_over_sdp}. The result itself also helps us gain some theoretical insight: it tells us that the worst-case for NoLips might be reached as $(f,h)$ approach possibly pathological limiting nonsmooth functions in $\blnonsmooth(\reals^n)$. 
% We observed this phenomenon empirically when trying to guess a worst-case Bregman method: this led us in turn to the design of the smoothed lower bound presented in Section \ref{ss:lower_bound}.

Observe also that we focused on presenting the PEP for the class $\mathcal{B}_L(\reals^n)$ to avoid technicalities related to the domain of definition. However, we show in Appendix \ref{a:domh} that the exact same problem \eqref{eq:pep_over_sdp} also solves the performance estimation problem for NoLips on the general class $\mathcal{B}_L(C)$, for any nonempty closed convex set $C$.

\subsection{Numerical evidence and computer-assisted proofs}\label{ss:proofs}

We now provide several applications of the performance estimation framework that we developed for Bregman methods. 
% We then also extend it for the inertial variant IGA \cite{Auslender2006}...

\subsubsection{Solving \eqref{eq:pep_nolips} for finding the exact worst-case convergence rate of NoLips}\label{sss:nolips_pep}

We first start by the most direct application, that is finding the exact worst-case performance of NoLips. Theorem \ref{thm:equiv} states that it can be computed by solving the semidefinite program \eqref{eq:pep_over_sdp}. The link to the MATLAB implementation is provided in Section \ref{s:conclusion}.

To simplify our setting, note that we can assume without loss of generality that the relative smoothness constant $L$ is $1$, since we can replace $h$ by a scaled version $Lh$. Recall that we know from Theorem \ref{thm:nolips_bound}, that
\[ 
\text{val\eqref{eq:pep_nolips}} \leq \frac{1}{\lambda N}.
\]
Table \ref{tab:pep_nolips} shows the result of solving \eqref{eq:pep_over_sdp} for several values of $N$ up to 100, for a step size $\lambda = 1$. We observe that with high precision, val\eqref{eq:pep_over_sdp} is equal to the theoretical bound $1/(\lambda N)$.

\begin{table}[H]
\footnotesize
\begin{center}
\caption{ Numerical value of the performance estimation problem \eqref{eq:pep_nolips} with $\lambda = 1$, $L = 1$. \textit{Rel. error} denotes the relative error between val\eqref{eq:pep_nolips} and the theoretical bound of $1/N$ given by Theorem \ref{thm:nolips_bound}. \textit{Primal feasibility} corresponds to the maximal absolute value of constraint violation returned by the MOSEK solver.}
\label{tab:pep_nolips}       
\begin{tabular}{lccc}
\hline\noalign{\smallskip}
N & val\eqref{eq:pep_nolips} & Rel. error & Primal feasibility  \\
\noalign{\smallskip}\hline\noalign{\smallskip}
1 & 1.000 & 1.8e-11 &  4.3e-10 \\
2 & 0.500 & 1.8e-8 & 2.8e-9\\
3 & 0.333 & 1.8e-8 & 2.8e-9 \\
4 & 0.250 & 4.9e-8 & 2.3e-8 \\
5 & 0.200 & 1.8e-10 & 6.4e-11 \\
10 & 0.100 & 6.4e-11 & 1.3e-11 \\
20 & 0.050 & 1.1e-8 & 1.9e-10\\
50 & 0.020 & 6.5e-6 & 5.0e-7\\
100 & 0.01 & 7.2e-5 & 1.6e-6\\
\noalign{\smallskip}\hline
\end{tabular}
\end{center}
\end{table}
\paragraph{Other values of $\lambda$.} One can wonder how the numerical value evolves when one varies the step size $\lambda$. Our experimental observations are as follows:
\begin{itemize}
    \item For any $\lambda \in (0,1/L]$, val\eqref{eq:pep_nolips} is exactly equal to the theoretical bound $1/(\lambda N)$.
    \item For any $\lambda > 1/L$, val\eqref{eq:pep_nolips} $= +\infty$, hence Algorithm \ref{algo:bpg} does not converge in general with these step size values. This suggests that the maximal step size value allowed for NoLips is indeed $1/L$, unlike the Euclidean setting where gradient descent can be applied with a step size that goes up to $2/L$.
\end{itemize}

While results above suggest that $1/(\lambda N)$ is the exact worst-case rate of NoLips, they provide only numerical evidence. We can however use them to deduce formal guarantees, both for proving an \textit{upper bound} and a \textit{lower bound}.

\paragraph{Upper bound guarantee through duality.} As noticed in previous work on PEPs \cite{Drori2014,Taylor2015}, solving the dual of \eqref{eq:pep_over_sdp} can be used to deduce a proof. Indeed, the dual solution gives a combination of the constraints that, when transposed to analytical form, leads to a formal guarantee. 
This provides the following proof for the $O(1/k)$ convergence rate of Theorem \ref{thm:nolips_bound}.
\paragraph{Proof of Theorem \ref{thm:nolips_bound}} The proof relies on the fact that, since $Lh-f$ is convex we have that $\tfrac{1}{\lambda}h-f$ is convex for any $\lambda\in (0,\tfrac1L]$, and only consists in performing the following weighted sum of inequalities:
\begin{itemize}
    \item convexity of $f$, between $u$ and $x_i$  ($i=0,\hdots,k$) with weights $\gamma_{*,i}=\tfrac{1}{k}$:
    \[f(u)\geq f(x_i)+ \inner{ \nabla f(x_i)}{u-x_i},\]
    \item convexity of $f$, between $x_i$ and $x_{i+1}$ ($i=0,\hdots,k-1$) with weights $\gamma_{i,i+1}=\tfrac{i}{k}$:
    \[f(x_i)\geq f(x_{i+1})+\inner{\nabla f(x_{i+1})}{x_i-x_{i+1}},\]
    \item convexity of $\tfrac{1}{\lambda}h-f$, between $u$ and $x_k$ with weight $\mu_{*,k}=\tfrac1k$:
    \[\tfrac{1}{\lambda}h(u)-f(u)\geq \tfrac{1}{\lambda}h(x_k)-f(x_k)+\inner{\tfrac{1}{\lambda}\nabla h(x_k)-\nabla f(x_k)}{u-x_k},\]
    \item convexity of $\tfrac{1}{\lambda}h-f$, between $x_{i+1}$ and $x_{i}$ ($i=0,\hdots,k-1$) with weight $\mu_{i+1,i}=\tfrac{i+1}{k}$
    \[\tfrac{1}{\lambda}h(x_{i+1})-f(x_{i+1})\geq \tfrac{1}{\lambda}h(x_i)-f(x_i)+\inner{\tfrac{1}{\lambda}\nabla h(x_i)-\nabla f(x_i)}{x_{i+1}-x_i},\]
    \item convexity of $\tfrac{1}{\lambda}h-f$, between $x_{i}$ and $x_{i+1}$ ($i=0,\hdots,k-1$) with weight $\mu_{i,i+1}=\tfrac{i}{k}$
    \[\tfrac{1}{\lambda}h(x_i)-f(x_i)\geq \tfrac{1}{\lambda}h(x_{i+1})-f(x_{i+1})+\inner{\tfrac{1}{\lambda}\nabla h(x_{i+1})-\nabla f(x_{i+1})}{x_i-x_{i+1}}.\]
\end{itemize}
The weighted sum is written as
\begin{equation*}
\begin{aligned}
0\geq &\sum_{i=0}^{k}\gamma_{*,i}\left[f(x_i)-f(u)+ \inner{ \nabla f(x_i)}{u-x_i}\right]\\
&+\sum_{i=0}^{k-1} \gamma_{i,i+1} \left[ f(x_{i+1})-f(x_i)+\inner{\nabla f(x_{i+1})}{x_i-x_{i+1}}\right]\\
&+ \mu_{*,k} \left[\tfrac{1}{\lambda}h(x_k)-f(x_k)-(\tfrac{1}{\lambda}h(u)-f(u))+\inner{\tfrac{1}{\lambda}\nabla h(x_k)-\nabla f(x_k)}{u-x_k} \right]\\
&+\sum_{i=0}^{k-1} \mu_{i+1,i} \left[\tfrac{1}{\lambda}h(x_i)-f(x_i)-( \tfrac{1}{\lambda}h(x_{i+1})-f(x_{i+1}))+\inner{\tfrac{1}{\lambda}\nabla h(x_i)-\nabla f(x_i)}{x_{i+1}-x_i}\right]\\
&+\sum_{i=0}^{k-1} \mu_{i,i+1} \left[\tfrac{1}{\lambda}h(x_{i+1})-f(x_{i+1})-(\tfrac{1}{\lambda}h(x_i)-f(x_i))+\inner{\tfrac{1}{\lambda}\nabla h(x_{i+1})-\nabla f(x_{i+1})}{x_i-x_{i+1}}\right].
\end{aligned}
\end{equation*}
By substitution of $\nabla h(x_{i+1})=\nabla h(x_{i})-\lambda \nabla f(x_{i})$ ($i=0,\hdots,k-1$), one can reformulate the weighted sum exactly as (i.e., there is no residual): 
\[0\geq f(x_k)-f(u)- \tfrac{h(u)-h(x_0)-\inner{\nabla h(x_0)}{u-x_0}}{\lambda k},\]
yielding the desired result. 
\qed

\paragraph{Lower bound through worst-case functions.} As \eqref{eq:pep_nolips} computes the \textit{exact} worst-case performance of NoLips, experiments above suggest that $1/(\lambda N)$ is also a lower bound, meaning that for every $\epsilon > 0$, there exist functions $(f,h) \in \mathcal{B}_L$ such that the iterates of NoLips satisfy 
\[f(x_N) - f_* \geq \frac{D_h(x_*,x_0)}{\lambda N} - \epsilon.\]
We detail here how such functions can be constructed from the solution of \eqref{eq:pep_over_sdp}. The numerical solver allows us to find a maximizer $\overline{G},\overline{F},\overline{H}$ (recall that only the relaxed problem has a maximizer as the feasible set is closed), and by factorizing the matrix $G$ as $P^T P$, we can thus recover the corresponding discrete representation $\{\overline{x}_i,\overline{g}_i,\overline{f}_i,\overline{h}_i,\overline{s}_i\}_{i \in I}$. This discretization can in turn be interpolated to get the corresponding functions $(\overline{f},\overline{h}) \in \blnonsmooth$. There are multiple ways to perform this interpolation; see \cite[Thm. 1]{Taylor2017} for a constructive approach.

Recall that since functions $(\overline{f},\overline{h})$ yield a  solution to \eqref{eq:pep_overestimation}, they belong to $\blnonsmooth$ and might thus form a \textit{pathological} nonsmooth limiting worst-case. They can be approached by valid instances $(f_\mu,h_\mu) \in \mathcal{B}_L$ by performing for instance smoothing through Moreau evelopes (as in Section \ref{ss:lower_bound}) and adding a small quadratic to $h$ to make it strictly convex.

There are however many possible maximizers of \eqref{eq:pep_over_sdp}. If we seek a low-dimensional example that may be easily interpretable, we can search for a maximizer such that the Gram matrix $G$ has minimal rank. Using rank minimization heuristics, we were able to find one-dimensional worst-case functions. Fix a number of iterations $N \geq 1$, assume $\lambda = 1/L = 1$ and define $\overf,\overh:\reals \rightarrow \reals$ as
\[
\begin{split}
    \overf(x) &= |x-1|,\\
    \overh(x) &= \overf(x) + \max(- N x,0),
\end{split}
\]
and set $\overx_0 = 0, \overx_* = 1$. Then clearly $(\overf,\overh) \in \blnonsmooth(\reals)$. Figure \ref{fig:worstcase1d} shows the functions $\overf,\overh$ as well as their smoothed versions $(f_\mu,h_\mu) \in \mathcal{B}_L(\reals)$. Note that the pathological behavior also reflects in the iterates: in the limiting instance, all iterates $\overx_0,\dots,\overx_N$ are equal. In the smoothed version, iterates are distinct (since $h_\mu$ is strictly convex), but they get closer and closer as the smoothing parameter $\mu$ goes to $0$.

The smoothed function $f_\mu$ is a Huber function, which is also the worst-case instance for Euclidean gradient descent on $L$-smooth functions described in \cite{Taylor2017}.
This analysis could be formalized to prove the $1/k$ lower bound for NoLips; however, this bound is just a particular case of the stronger result for general Bregman gradient methods derived in Section \ref{ss:lower_bound}.
\input{example_lower_bound_nolips.tex}
\subsubsection{Extension to other criteria}\label{sss:other_bound}

In our performance estimation problem, we focused on studying bounds of the form $f(x_N)-f_*\leq \theta(N,L,\lambda) D_h(x_*,x_0)$. However, we are not limited to this criterion, and different convergence measures might be considered by changing the objective and constraints in \eqref{eq:pep_nolips}. For instance, another popular criterion is the stationarity measure $D_h(x_k,x_{k+1})$, which boils down to the squared gradient norm in the unconstrained Euclidean case. By adapting \eqref{eq:pep_nolips}, we get the following new convergence result for NoLips.

\begin{proposition}[NoLips convergence rate, take II]\label{prop:nolips_bound2}
    Let $L > 0$, $C$ be a nonempty closed convex subset of $\reals^n$ and $(f,h) \in \mathcal{B}_L(C)$ a relatively-smooth problem instance. Then the sequence $\{x_k\}_{k \geq 0}$ generated by Algorithm \ref{algo:bpg} with constant step size $\lambda \in (0,1/L]$ satisfies for $k \geq 2$
    
    \begin{equation*}
        \min_{1\leq i\leq k} D_h(x_{i-1},x_i) \leq \frac{2D_{h}(x_*,x_0)}{k(k-1)}
    \end{equation*}
    for every $x_* \in \argmin_{C} f \cap \dom h$.
\end{proposition}
\begin{proof} In the same way as before, the formal guarantee has been obtained by examining the dual of the corresponding PEP.
The  proof relies on the fact that $\tfrac{1}{\lambda}h-f$ is convex for any $\lambda\in (0,\tfrac1L]$, and only consists in performing the following weighted sum of inequalities:
\begin{itemize}
    \item convexity of $f$, between $x_*$ and $x_i$  ($i=0,\hdots,k$) with weights $\gamma_{*,i}=\tfrac{2\lambda}{k(k-1)}$:
    \[f(x_*)\geq f(x_i)+ \inner{ \nabla f(x_i)}{x_*-x_i},\]
    \item optimality of $x_*$ for each $x_k$ with weight $\gamma_{k,*}=\tfrac{2\lambda}{k-1}$:
    \[f(x_k)\geq f(x_*),\]
    \item convexity of $\tfrac{1}{\lambda}h-f$, between $x_*$ and $x_k$ with weight $\mu_{*,k}=\tfrac{2\lambda}{k(k-1)}$:
    \[\tfrac{1}{\lambda}h(x_*)-f(x_*)\geq \tfrac{1}{\lambda}h(x_k)-f(x_k)+\inner{\tfrac{1}{\lambda}\nabla h(x_k)-\nabla f(x_k)}{x_*-x_k},\]
    \item convexity of $\tfrac{1}{\lambda}h-f$, between $x_{i+1}$ and $x_{i}$ ($i=0,\hdots,k-1$) with weight $\mu_{i+1,i}=\tfrac{2\lambda(i+1)}{k(k-1)}$
    \[\tfrac{1}{\lambda}h(x_{i+1})-f(x_{i+1})\geq \tfrac{1}{\lambda}h(x_i)-f(x_i)+\inner{\tfrac{1}{\lambda}\nabla h(x_i)-\nabla f(x_i)}{x_{i+1}-x_i},\]
    \item definition of smallest residual among the iterates ($i=1,\hdots,k$) with weights $\tau_i=\tfrac{2(i-1)}{k(k-1)}$:
    \[ h(x_{i-1})-h(x_i)-\inner{\nabla h(x_i)}{x_{i-1}-x_i}\geq \min_{1\leq j \leq k} \{D_h(x_{j-1},x_j)\}.\]
\end{itemize}
The weighted sum is written as
\begin{equation*}
\begin{aligned}
0\geq & \sum_{i=0}^k \gamma_{*,i} [f(x_i)-f(x_*)+ \inner{ \nabla f(x_i)}{x_*-x_i}]\\
&+ \gamma_{k,*} [f(x_*)-f(x_k)]\\
&+ \mu_{*,k} [\tfrac{1}{\lambda}h(x_k)-f(x_k)-(\tfrac{1}{\lambda}h(x_*)-f(x_*))+\inner{\tfrac{1}{\lambda}\nabla h(x_k)-\nabla f(x_k)}{x_*-x_k}]\\
&+\sum_{i=0}^{k-1} \mu_{i+1,i} [\tfrac{1}{\lambda}h(x_i)-f(x_i)-(\tfrac{1}{\lambda}h(x_{i+1})-f(x_{i+1}))+\inner{\tfrac{1}{\lambda}\nabla h(x_i)-\nabla f(x_i)}{x_{i+1}-x_i}]\\
&+ \sum_{i=1}^k \tau_i [\min_{1\leq j \leq k} \{D_h(x_{j-1},x_j)\} -( h(x_{i-1})-h(x_i)-\inner{\nabla h(x_i)}{x_{i-1}-x_i})].
\end{aligned}
\end{equation*}
By substitution of $\nabla h(x_{i+1})=\nabla h(x_{i})-\lambda \nabla f(x_{i})$ ($i=0,\hdots,k-1$), one can reformulate the weighted sum exactly as (i.e., there is no residual): 
\[0\geq \min_{1\leq j \leq k} \{D_h(x_{j-1},x_j)\} - 2\tfrac{h(x_*)-h(x_0)-\inner{\nabla h(x_0)}{x_*-x_0}}{ k(k-1)},\]
yielding the desired result.\qed
\end{proof}

\subsubsection{Beyond NoLips: inertial Bregman algorithms}\label{ss:iga}

\begin{figure}[t]
% 		\begin{center}
        \centering
			\begin{tikzpicture}
			\begin{loglogaxis}[legend pos=south west, legend style={draw=none, inner sep = 0pt},legend cell align={left}, plotOptions,  ymin=0.0001, ymax=10,xmin=1,xmax=100,width=.4\linewidth, 
	       width=0.5\textwidth,height = 0.4\textwidth
			]
			\addplot [colorP1] table [x=k,y=WC_FUNC]  {Data/NoLips_2.dat};
% 			\addplot [colorP2] table [x=k,y=WC_GRAD]  {Data/NoLips.dat};
			\addplot [colorP3,dashed] table [x=k,y=WC_MINGRAD] {Data/NoLips_2.dat};
			\addlegendentry{$({f(x_k)-f(x_*)})/{D_h(x_*,x_0)}$}
% 			\addlegendentry{${D_h(x_{k-1},x_k)}/{D_h(x_*,x_0)}$}
			\addlegendentry{$\min_{1\leq i\leq k} \{D_h(x_{i-1},x_i)\}/{D_h(x_*,x_0)}$}
			\end{loglogaxis}
			\end{tikzpicture}
			\hspace{1pt}
			\begin{tikzpicture}
			\begin{loglogaxis}[legend pos=south west, legend style={draw=none,inner sep = 0pt},legend cell align={left}, plotOptions,  ymin=0.0001, ymax=10,xmin=1,xmax=100,width=.4\linewidth, 
	       width=0.5\textwidth,height = 0.4\textwidth
			]
			\addplot [colorP2,dashed] table [x=N,y=theory]  {Data/IGA.dat};
			\addplot [colorP1] table [x=N,y=pep_iga] {Data/IGA.dat};
			\addlegendentry{$4\tilde{L}/(\sigma k^2) \, \text{(theory)}$}
% 			\addlegendentry{${D_h(x_{k-1},x_k)}/{D_h(x_*,x_0)}$}
			\addlegendentry{$({f(x_k)-f(x_*)})/{(D_h(x_*,x_0)+f(x_0)-f_*)}$}
			\end{loglogaxis}
			\end{tikzpicture}
						
 			\caption{ Numerical worst-case guarantees obtained from PEPs as functions of the iteration counter $k$ (shown in log scale as rates are sublinear). \textbf{Left:} guarantees for NoLips (Algorithm \ref{algo:bpg}) for two different convergence measures. Numerical values confirm exactly the theoretical rates of Theorem \ref{thm:nolips_bound} and Proposition \ref{prop:nolips_bound2}. \textbf{Right:} guarantees for IGA with no affine constraints (Algorithm \ref{algo:iga}) under the assumption that $h$ is $1$-strongly convex and $f$ is 1-smooth, compared to the theoretical bound from \cite{Auslender2006}. Notice that the theoretical bound is not tight in this case, as it is obtained by making some approximations in the proof.}\label{fig:pep_nolips}
% 		\end{center}
\end{figure}

%\paragraph{} 
Our approach is not limited to the NoLips algorithm. For instance, we can also solve the performance estimation problem for the inertial Bregman algorithm proposed by Auslender and Teboulle \cite{Auslender2006}, a.k.a. the Improved Interior Gradient Algorithm (IGA). We recall its simplified formulation in Algorithm \ref{algo:iga}, in the case where there are no affine constraints.

\begin{algorithm}[H]
{\normalsize
	\begin{algorithmic}
	\begin{spacing}{1.4}
		\STATE \textbf{Input:} Functions $f,h$, initial point $x_0 \in \interior \dom h$, step size $\lambda$.\\
		\STATE Set $z_0 = x_0$ and $t_0 = 1$.
		\FOR{k = 0,1,\dots}
		\STATE $y_k = (1-\frac{1}{t_k}) x_k + \frac{1}{t_k} z_k$
		\STATE 
		    $
		    z_{k+1} = \argmin \, \{\la \nabla f(y_k), u-y_k \ra + \frac{1}{t_k \lambda}D_h(u,z_k) \,|\, u \in \reals^n\}  $
		\STATE $x_{k+1} = (1 - \frac{1}{t_k}) x_k + \frac{1}{t_k} z_{k+1}$
		\STATE $t_{k+1} = (1 + \sqrt{1 + 4t_k^2})/2$.
		\ENDFOR \vspace{-3mm}
		\end{spacing}
	\end{algorithmic}
	\caption{Improved Interior Gradient Algorithm (IGA) \cite{Auslender2006}}
	\label{algo:iga}
}
\end{algorithm}

In the setting where $f$ has $\tilde{L}$-Lipschitz continuous gradients and $h$ is a $\sigma$-strongly convex kernel function, IGA with step size $\lambda = \sigma/\tilde{L}$ enjoys the following convergence rate \cite[Thm. 5.2]{Auslender2006}:
\begin{equation}\label{eq:iga_bound}
    f(x_N) - f_* \leq \frac{4 \tilde{L}}{\sigma N^2} \left(D_h(x_*,x_0) + f(x_0) - f_* \right).
\end{equation}
Our PEP framework can also be applied to this algorithm, in order to find the smallest value of $\theta(N,\tilde{L},\sigma,\lambda)$ which satisfies
\[f(x_N) - f_* \leq \theta(N,\tilde{L},\sigma,\lambda) \left(D_h(x_*,x_0) + f(x_0) - f_* \right) \] 
for every instance of IGA with the supplementary assumptions made above. In this case, we use the standard interpolation conditions of Theorem \ref{thm:interp_mul} for $L$-smooth and strongly convex functions.
Results are shown in Figure~\ref{fig:pep_nolips}. The exact numerical worst-case performance of IGA is slightly below the theoretical bound above, since the proof in \cite{Auslender2006} makes some approximations.

% \textit{IGA in the general relatively-smooth case: failure of inertia.} We pointed out in Section \ref{s:setup} that the setting in which $f$ is $\tilde{L}$-smooth and $h$ is $\sigma$-strongly convex is a particular case of $h$-smoothness with constant $L = \tilde{L}/\sigma$. The natural question that was also raised in \cite[Section 6]{Teboulle2018} is therefore: does IGA converge for the general class $\mathcal{B}_L(C)$ ? Solving the corresponding PEP yields the following results. For Algorithm \ref{algo:iga} with the setting that $(f,h) \in \mathcal{B}_L(C)$ and several choices of step size in $(0,1/L]$, the solver states the value of the corresponding performance estimation problem \textbf{is unbounded}, i.e., there does not exist any $\theta$ such that the bound~\eqref{eq:iga_bound} holds for every instance $(f,h) \in \bl$. Of course, this constitutes numerical evidence and not a formal proof. Nonetheless, due to the tightness result of Theorem \ref{thm:equiv}, there are strong reasons to conjecture that IGA indeed does not converge in the general $h$-smooth setting. 
 
\paragraph{IGA in the general relatively-smooth case: failure of inertia.} We pointed out in Section \ref{s:setup} that the setting in which $f$ is $\tilde{L}$-smooth and $h$ is $\sigma$-strongly convex is a particular case of relative smoothness with constant $L = \tilde{L}/\sigma$. The natural question that was also raised in \cite[Section 6]{Teboulle2018} is therefore: does IGA converge for the general class $\mathcal{B}_L(C)$ ? Solving the corresponding PEP yields the following results. For Algorithm \ref{algo:iga} with the setting that $(f,h) \in \mathcal{B}_L(C)$ and several choices of step size in $(0,1/L]$, the solver states that the value of the corresponding performance estimation problem \textbf{is unbounded}, i.e., there does not exist any $\theta$ such that the bound~\eqref{eq:iga_bound} holds for every instance $(f,h) \in \bl$.

\modif{
One could legitimately wonder whether there exist other sequences $\{t_k\}_{k\geq 0}$ with $t_k > 1$, perhaps \textit{less aggressive} than the one in Algorithm \ref{algo:iga}, such that the method converges (note that choosing $t_k = 1 \,\,\, \forall k \geq 0$ would yield the standard NoLips scheme). After solving the PEP with several choices of such sequences and observing that it is unbounded, we formulate the following conjecture: for \textit{any} sequence $\{t_k \}_{k \geq 0}$, in IGA, such that $t_{k_0} > 1$ for some $k_0$, it is not possible to bound $f(x_N)-f_*$ in general. Of course, this constitutes numerical evidence and not a formal proof. The conjecture could be proved by constructing worst-case functions in the same spirit as in Section \ref{s:complexity}, with some pathological lack of smoothness that would cause the iterates to diverge when taking a step size larger than $1/L$.

These experiments lead us to believe that inertial methods with non-adaptive coefficients fail to converge in the general relatively-smooth setting.
}

\subsubsection{From worst-case functions for NoLips to a lower bound for general Bregman methods}

We briefly explain how, with the PEP methodology, the worst case functions from Section \ref{ss:lower_bound} were discovered. 

We described in Section \ref{sss:nolips_pep} how a one-dimensional worst-case instance $(\overf,\overh)$ for NoLips was discovered from low-rank solutions of \eqref{eq:pep_over_sdp}. However, this instance may not be difficult enough for a more generic Bregman algorithm that can use abritrary linear combinations of gradients (as in Definition \ref{def:algorithm}, our definition of the \textit{Bregman gradient algorithm}), and thus cannot be used to prove a general lower bound.

Our objective now is to find worst-case instances that are difficult for \textbf{any} Bregman gradient algorithm. A desirable property would be that these instances allow to explore only \textit{one dimension} per oracle call, so that the function \textit{hides information} in the unexplored dimensions. This is  similar in spirit to the so-called ``worst function in the world" of Nesterov \cite{Nesterov2004}. In order to achieve this goal, we propose to search for functions $f$ for which all gradients $\nabla f(x_i)$ are orthogonal, guaranteeing that one new dimension is explored at each step. Note that a similar approach has been used in some previous work on PEPs to find lower bounds or optimal methods e.g., in \cite{Drori2017,Drori2019a}. This amounts to adding some orthogonality constraints to \eqref{eq:pep_nolips} and solving
\begin{equation}\label{eq:pep_nolips_orth}\tag{PEP-orth}
\BA{ll}
       \mbox{maximize} & \big(f\left(x_N\right) - f \left(x_*\right)\big) / D_{h}(x_*,x_0)\\
      \rule{0pt}{12pt}
        \mbox{subject to   } &(f,h) \in \mathcal{B}_L(\reals^n), \\
        % & Lh - f \mbox{ is convex,} \\
        & x_* \mbox{ is a minimizer of $f$}, \\
        & x_1,\dots,x_N \mbox{ are generated from $x_0$ by Algorithm \ref{algo:bpg} with step size $\lambda$,}\\
        & \la \nabla f(x_i), \nabla f(x_j) \ra = 0 \mbox{ for } i\neq j \in I,
\EA
\end{equation}
in the variables $f,h,x_0,\dots,x_N,x_*,n$.
% \begin{equation*}
%     \begin{split}
%         x_0 = (1,1),f_0 &= 1,\, g_0 = (1,0),\, h_0 = 1, \, s_0 = (1,1),\\
%         x_1 = (0,1),f_1 &= 1,\, g_1 = (0,1),\, h_1 = 1, \, s_0 = (0,1),\\
%         x_* = (0,0),f_* &= 0,\, g_* = (0,0),\, h_* = 0, \, s_* = (0,0).\\
%     \end{split}
% \end{equation*}

% From this discrete representation, convex functions $f,h$ can be interpolated. However, these functions will be \textbf{nonsmooth} ($g_i,s_i$ are in that case subgradients of $f,h$) and therefore not valid instances of NoLips. This is because the set $\mathcal{B}_L$ of relatively-smooth problems is an open set, and the maximizing sequence of \eqref{eq:pep_nolips} actually approaches limiting nonsmooth functions $(\overline{f}, \overline{h})$.
% From the values above, such limiting functions can be deduced for $N = 2$, indeed define for instance
In the same spirit as before, we were able to find a dimension-$N$ solution of \eqref{eq:pep_nolips_orth}. This allows us to interpolate the following worst-case pathological instance:
\begin{equation*}
    \begin{split}
        \overline{f}(x) &= \|x-(1,\dots, 1)\|_\infty,\\
        \overline{h}(x) &= \overline{f}(x) + \sum_{i=2}^N \max(-x^{(i)},0).
    \end{split}
\end{equation*}
Again, these are nonsmooth functions and, as such, they do not form valid instances for NoLips. However, they can be approached by a sequence of such functions, for instance by applying smoothing with the Moreau evelope, and adding a small quadratic term to make $h$ strictly convex. Along with a few tweaks, this is how we found the example that was used to prove the general lower bound for $\mathcal{B}_L$ in Section \ref{ss:lower_bound}.

\section{Conclusion}\label{s:conclusion}

%  - methodology easily applicable to other Bregman algorithms
%  - If the domain is open, complexity results are out of reach (cf counter example), and if it is open, it can be reduced to studying functions with full domain 
%  - Worst case behavior is a non-differentiable limit
%  - In order to get acceleration, need for additional assumptions, weaker than L-smoothness

Our paper has two main contributions: proving optimality of NoLips for the general relatively-smooth setting, and developing numerical performance estimation techniques for Bregman gradient algorithms.
We presented the performance estimation problem on the basic NoLips algorithm for simplicity, but our approach can be applied to different settings and various algorithms involving Bregman distances. We provided several applications illustrating how the PEP methodology is an efficient tool for conjecturing and analyzing the worst-case behavior of Bregman algorithms.

There is a fundamental concept linking the two parts of the paper, which is that of \textit{limiting nonsmooth pathological behavior}. 
When looking for worst-case guarantees over a class of functions that is open such as the class of differentiable convex functions, the performance estimation problem is a \emph{supremum} and the worst-case maximizing \textit{sequence} might approach some function that is not in this class, e.g., one that is nonsmooth in our case. This idea, observed by analyzing the equivalence between \eqref{eq:pep_nolips} and the nonsmooth relaxation \eqref{eq:pep_overestimation}, was used in the proof of the lower bound in Section \ref{ss:lower_bound}. Moreover, the worst-case sequence of functions was directly inspired by examining particular solutions of \eqref{eq:pep_overestimation}. 

Our result also shows that additional assumptions on functions $f$ and $h$ are needed in order to prove better bounds or devise faster algorithms than NoLips. If the usual properties of $L$-smoothness and strong convexity are too restrictive and do not hold in many applications, the future challenge is to find weaker assumptions, that define a larger class of functions where improved rates can be obtained. One other possible approach would be to find algorithms that do not fit in Definition \ref{def:algorithm}, for instance by including second-order oracles of $h$, in the case when  $h$ is simple enough.

% \newpage
% \begin{acknowledgements}
\bigskip
\bigskip
\bigskip
% {\small

\textbf{Code.} Experiments have been run in MATLAB, using the semidefinite solver MOSEK \cite{mosek} as well as the modeling toolbox YALMIP \cite{yalmip}. The support for Bregman methods has been added to the Performance Estimation Toolbox (PESTO, \cite{Taylor2018}) for which we provide some examples.
The code can be downloaded from
{\small \url{https://github.com/RaduAlexandruDragomir/BregmanPerformanceEstimation}}.

\bigskip

\textbf{Acknowledgements.} The authors would like to thank the anonymous reviewers for constructive suggestions as well as Dmitrii Ostrovskii and Edouard Pauwels for useful comments. RD acknowledges support from an AMX fellowship. AT acknowledges support from the European Research Council (grant SEQUOIA 724063).
AA is at CNRS, and CS Department, Ecole Normale Sup\'erieure, PSL Research University, 45 rue d'Ulm, 75005, Paris. AA would like to acknowledge support from the {\em ML and Optimisation} joint research initiative with the {\em fonds AXA pour la recherche} and Kamet Ventures, a Google focused award, as well as funding by the French government under management of Agence Nationale de la Recherche as part of the "Investissements d'avenir" program, reference ANR-19-P3IA-0001 (PRAIRIE 3IA Institute). JB acknowledges the support of ANR-3IA ANITI, ANR Chess, Air Force Office of Scientific Research, Air Force Material Command, USAF, under grant numbers FA9550-19-1-7026, FA9550-18-1-0226. JB acknowledges financial support of the research foundation TSE-Partnership.
% \end{acknowledgements}

% }

% \newpage

% Authors must disclose all relationships or interests that 
% could have direct or potential influence or impart bias on 
% the work: 
%
% \section*{Conflict of interest}
%
% The authors declare that they have no conflict of interest.

 \appendix
% \section{Proof of Theorem \ref{thm:nolips_bound}} \label{a:proof_bound_nolips}

% \section{Proof of Proposition \ref{prop:nolips_bound2}} \label{a:proof_bound_nolips2}

\section{Extension of performance analysis to the case when $C$ is a general closed convex subset of $\reals^n$}\label{a:domh}
For simplicity of the presentation, we left out in Section \ref{s:pep} the case when  the domain $C$ is a proper subset of $\reals^n$. We show in this section that it actually corresponds to the same minimization problem \eqref{eq:pep_over_sdp}.

Let us formulate the performance estimation problem for Algorithm \ref{algo:bpg} in the general case. Recall that we denote $\mathcal{B}_L$ the union of $\mathcal{B}_L(C)$ for all closed convex subsets of $\reals^n$ and  for every $n \geq 1$. The performance estimation problem writes
\begin{equation}\label{eq:pep_nolips_c}\tag{PEP-C}
\BA{ll}
       \mbox{maximize} & \big(f\left(x_N\right) - f \left(x_*\right)\big) / D_{h}(x_*,x_0)\\
       \rule{0pt}{12pt}
        \mbox{subject to   } &(f,h) \in \mathcal{B}_L, \\
        % & Lh - f \mbox{ is convex,} \\
        & x_* \mbox{ is a minimizer of $f$ on } \overline{\dom h} \mbox{ such that } x \in \dom h, \\
        & x_1,\dots,x_N \mbox{ are generated from $x_0$ by Algorithm \ref{algo:bpg} with step size $\lambda$,}\\
\EA
\end{equation}
in the variables $f,h,x_0,\dots,x_N,x_*,n$.
Now, as \eqref{eq:pep_nolips_c} is a problem that includes \eqref{eq:pep_nolips} in the special case where $C = \reals^n$, its value is larger:
\[ \mbox{val\eqref{eq:pep_nolips}} \leq  \mbox{val\eqref{eq:pep_nolips_c}} 
\]
Let us show that val\eqref{eq:pep_nolips_c} is upper bounded by the same relaxation val\eqref{eq:pep_overestimation}, which allows to conclude that the values are equal.
We recall that the problem \eqref{eq:pep_overestimation} can be written, using interpolation conditions of Corollary \ref{cor:blnonsmooth}, as
\begin{equation}\tag{$\overline{\text{PEP}}$}
\BA{rcl}
        &\mbox{maximize} & f_N - f_* \\
        \rule{0pt}{11pt}
        &\mbox{subject to} &  f_i - f_j - \la g_j, x_i - x_j \ra \geq 0,\\
        && (Lh_i - f_i) - (Lh_j - f_j) - \la Ls_j - g_j,x_i - x_j \ra \geq 0 \quad \mbox{ for } i,j \in I,\\
        &            & g_* = 0,\\
        &            & s_{i+1} = s_i - \lambda g_i \quad \mbox{for } i\in \{1,\dots, N-1\},\\
        &            &  h_* - h_0  - \la s_0, x_* - x_0 \ra = 1,\\
\EA
\end{equation}
in the variables $n,\{(x_i,f_i,g_i,h_i,s_i)\}_{i \in I}$. 
We show that every admissible point of \eqref{eq:pep_nolips_c} can be cast into an admissible point of \eqref{eq:pep_over_sdp}. This actually amounts to show that, from the point of view of performance estimation, an instance $(f,h) \in \bl(C)$ is actually equivalent to some instance in $\bl(\reals^n)$. 

Let $f,h,x_0,\dots,x_N,x_*$ be a feasible point of \eqref{eq:pep_nolips_c}. We distinguish two cases.

\paragraph{Case 1: \footnotesize{$x_* \in \interior \dom h$}.} This is the simplest case, as the necessary conditions are the same as in the situation where $C = \reals^n$. Indeed, then we have $x_0,\dots,x_N,x_* \in \interior \dom h$, since $x_0$ is constrained to be in the interior and the next iterates are in $\interior \dom h$ by Assumption \ref{assumption_bpg}. Since $f$ and $h$ are differentiable on $\interior \dom h$, convexity of $f$ and $Lh-f$ imply that the first two constraints of \eqref{eq:pep_overestimation} hold for all $i,j \in I$. Finally, $g_* = 0$ follows from the fact that $x_*$ minimizes $f$ and that it lies on the interior of the domain. Hence the discrete representation satisfies the constraints of \eqref{eq:pep_over_sdp}.

\paragraph{Case 2: \footnotesize{$ x_* \in \partial \dom h$}.}
 In this case, $f$ and $h$ are not necessarily differentiable at $x_*$, but are still differentiable still at $x_0,\dots,x_N$ for the same reasons. But we can still, with a small modification at $x_*$, derive a discrete representation that fits the constraints of \eqref{eq:pep_overestimation} and whose objective is the same. Indeed, define

\begin{equation*}
    \begin{split}
        (g_i,f_i,s_i,h_i) &= \left(\nabla f(x_i), f(x_i), \nabla h(x_i), h(x_i)\right) \mbox{     for } i = 0,\dots, N,\\
        (g_*,f_*,s_*,h_*) &= \left(0,f(x_*),v,h(x_*)\right),
    \end{split}
\end{equation*}
where $v \in \reals^n$ is a vector that are specified later.
Then, for $i \in I$ and $j \in \{0\dots N\}$, convexity of $f$ and $Lh-f$ imply that the constraints 
\[ \begin{split}
f_i - f_j - \la g_j, x_i- x_j \ra \geq 0\\
 (Lh_i-f_i) - (Lh_j-f_j) - \la Ls_j - g_j, x_i- x_j \ra \geq 0
 \end{split}\]
hold. It remains to verify them for $i \in \{0 \dots N\}$ and $j = *$. The first one holds because $x_*$ minimizes $f$ on $\dom h$, so with $g_* = 0$ we have $f_i - f_* \geq 0$. We now show that the second one is satisfied, i.e., that we can choose $v \in \reals^n$ so that 
\[(Lh_i-f_i) - (Lh_*-f_*) - \la Lv, x_i- x_* \ra \geq 0 \quad \forall i \in \{0\dots N\}.\]

To this extent, we use the fact that $x_* \in \partial \dom h$ and that $x_i \in \interior \dom h$ for $i = 0 \dots N$. This means that $\{x_*\} \cap \interior \dom h = \emptyset$, and therefore by the hyperplane separation theorem \cite[Thm 11.3]{Rockafellar2008}, there exists a hyperplane that separates the convex sets $\{x_*\}$ and $\interior \dom h$ \textit{properly}, meaning that there exists a vector $u \in \reals^n$ such that
\[\la x_i - x_*, u \ra < 0 \,\,\, \forall i \in \{0, \dots, N\}.\]
Set 
\[\begin{split} \alpha &= \min_{i =0\dots N}\, (L h_i - f_i) - (Lh_* - f_*),\\
\beta &= \min_{i = 0,\dots, N} - \la x_i - x_*, u \ra > 0,\end{split} \]

where $\beta > 0$ because of the separation result. Choose $s_* = v$ as $v = \frac{|\alpha|}{L \beta} u$.
Then we have
\[ \begin{split}
    (Lh_i - f_i) - (Lh_* - f_*) - \la L s_*,x_i - x_* \ra &\geq \alpha + L \frac{|\alpha|}{L \beta} \beta \\
    & \geq \alpha + |\alpha| \\& \geq 0.
\end{split}
\]
This eventually provides an instance $\{(x_i,g_i,f_i,h_i,s_i)\}_{i \in I}$ that is admissible for \eqref{eq:pep_overestimation}.

To conclude, we proved that in both cases, an admissible point of \eqref{eq:pep_nolips_c} can be turned into an admissible point of \eqref{eq:pep_over_sdp} with the same objective value. Hence we have
\[\mbox{val\eqref{eq:pep_nolips_c}} \leq \mbox{val\eqref{eq:pep_over_sdp}} .\]
Recalling that $\mbox{val\eqref{eq:pep_nolips}} \leq \mbox{val\eqref{eq:pep_nolips_c}}$ and that $\mbox{val\eqref{eq:pep_over_sdp}} = \mbox{val\eqref{eq:pep_nolips}}$ by Theorem \ref{thm:equiv}, we get
\[ \mbox{val\eqref{eq:pep_nolips_c}} = \mbox{val\eqref{eq:pep_nolips}}.\]
In other words, solving the performance estimation problem \eqref{eq:pep_nolips_c} for functions with any closed convex domain is equivalent to solving the performance estimation problem \eqref{eq:pep_nolips} restricted to functions that have full domain.

\normalsize
% BibTeX users please use one of
% \bibliographystyle{spbasic}      % basic style, author-year citations
\bibliographystyle{spmpsci}      % mathematics and physical sciences
\bibliography{library}   % name your BibTeX data base

\end{document}

%% file: defs.tex
\newtheorem{assumption}{Assumption}
\let\la\langle
\let\ra\rangle
\newcommand{\bl}{\mathcal{B}_L}
\newcommand{\blstrict}{\underline{\mathcal{B}_L}}
\newcommand{\blnonsmooth}{\overline{\mathcal{B}_L}}
\newcommand{\overf}{\overline{f}}
\newcommand{\overh}{\overline{h}}
\newcommand{\overx}{\overline{x}}

\newcommand{\BEAS}{\begin{eqnarray*}}
\newcommand{\EEAS}{\end{eqnarray*}}
\newcommand{\BEA}{\begin{eqnarray}}
\newcommand{\EEA}{\end{eqnarray}}
\newcommand{\BEQ}{\begin{equation}}
\newcommand{\EEQ}{\end{equation}}
\newcommand{\BIT}{\begin{itemize}}
\newcommand{\EIT}{\end{itemize}}
\newcommand{\BNUM}{\begin{enumerate}}
\newcommand{\ENUM}{\end{enumerate}}
\newcommand{\BPM}{\begin{pmatrix}}
\newcommand{\EPM}{\end{pmatrix}}

\newcommand{\BA}{\begin{array}}
\newcommand{\EA}{\end{array}}

\newcommand{\reals}{{\mathbb R}}

\newcommand{\symm}{{\mbox{\bf S}}}  % symmetric matrices

\newcommand{\QED}{~~\rule[-1pt]{6pt}{6pt}}\def\qed{\QED}

\newcommand{\argmin}{\mathop{\rm argmin}}

\newcommand{\dom}{\textup{dom}\,}
\newcommand{\interior}{\textup{int}\,}

\newcommand{\argmax}{\mathop{\rm argmax}}

\newcommand{\inner}[2]{{\langle #1, #2\rangle}}

\newcommand{\val}{\textup{val}}

\pgfplotsset{plotOptions/.style={%
		width=\linewidth,
		xlabel={Iteration $k$},
		label style={font=\scriptsize},
		legend style={font=\scriptsize},
		xtick={1,10,100},
		tick label style={font=\scriptsize},
		solid,
		very thick
	}}
\definecolor{colorP1}{RGB}{55,126,184}  % blue
		\definecolor{colorP2}{RGB}{228,26,28}  % red
		\definecolor{colorP3}{RGB}{152,78,163} % purple
		\definecolor{colorP4}{RGB}{77,175,74}  % green
		\definecolor{colorP5}{rgb}{0.6, 0.4, 0.08} % golden yellow

%% file: example_lower_bound_nolips.tex
% \begin{center}
	\begin{figure}[t]
	\centering
	        \pgfmathsetmacro{\Niter}{3.0}   % on définit des macros...
			\pgfmathsetmacro{\rangefm}{-0.8}% xmin du dessin
			\pgfmathsetmacro{\rangefp}{1.5} % xmax du dessin 
			\pgfmathsetmacro{\shift}{0.5}   % shift entre f et h 
			\pgfmathsetmacro{\smoothing}{0.12}   % smoothing parameter
			\pgfmathsetmacro{\strongconvexity}{0.1}   % strongconvexity parameter of h
            \newcommand{\colorh}{purple}    
            \pgfmathsetmacro{\figscale}{0.85}

			\begin{tikzpicture}[scale = \figscale]
			
			\begin{axis}[axis y line=none,
			             axis x line=bottom,
			            xtick={0,1},
	                    xticklabels={$\overx_0=\dots=\overx_3$,$\overx_*$},
	                    every tick/.style={very thick},
                      tick label style={font=\small}]
			
			\addplot [black, domain=\rangefm:\rangefp,samples=200, very thick] {max(x-1,-(x-1)} ;
			\addplot [\colorh, domain=\rangefm:\rangefp,samples=200, very thick] {\shift+max(x-1,-(x-1))+max(-\Niter*x,0)};
			\end{axis}
			\node[color=black, right, below] (f) at (1,2.8) {$\overf(x)$};
			\node[color=\colorh, right, above] (h) at (1,3.8+\shift) {$\overh(x)$};
			\draw [dotted] (2.37,0) -- (2.37,1.9);
			\draw [dotted] (5.37,0) -- (5.37,1);
			\end{tikzpicture}
			\hspace{4em}
			\begin{tikzpicture}[scale = \figscale]
			\begin{axis}[axis y line=none,
			             axis x line=bottom,
			            xtick={-0.48,-0.36,-0.24,-0.12,1},
	                    xticklabels={$x_0$,$x_1$,$x_2$,$x_3$,$x_*$},
	                    every tick/.style={very thick,},
                      tick label style={font=\small}]
			
			% smoothed objective : huber function
			 
			\addplot [black, domain=\rangefm:\rangefp,samples=200, very thick] {
			    (abs(x-1) < \smoothing) * (x-1)^2 / (2 * \smoothing) +
			    (abs(x-1) >= \smoothing) * (abs(x-1) - \smoothing / 2)
			} ;
			
			\addplot [\colorh, domain=\rangefm:\rangefp,samples=200, very thick] {
			    \shift+ (abs(x-1) < \smoothing) * (x-1)^2 / (2 * \smoothing) +
			    (x-1 >= \smoothing) * (abs(x-1) - \smoothing / 2) +
			    (x <= 1 -\smoothing) * (x >= -\smoothing) * (abs(x-1) - \smoothing / 2) +
			    (x <= -\smoothing * (\Niter + 1)) * (- (\Niter + 1) * x 
			        - 0.5 * \smoothing * (\Niter + 1)^2 + 1) +
			   (x >= -\smoothing * (\Niter + 1)) * (x <= -\smoothing) * (1 + 0.5 * x^2 / \smoothing) + \strongconvexity * x^2        
			} ;
			\end{axis}
			\node[color=black, right, below] (f) at (0.95,3.05) {$f_\mu(x)$};
			\node[color=\colorh, right, above] (h) at (1.05,3.62+\shift) {$h_\mu(x)$};
			\draw [dotted] (5.37,0) -- (5.37,1.2);
			\end{tikzpicture}
		\caption{ Worst-case functions for NoLips in dimension 1 with $N = 3$ iterations. The left figure shows the limiting instance $(\overf,\overh)\in \blnonsmooth(\reals)$, while the right plot represents the smooth approximation by a valid instance $(f_\mu,h_\mu)\in \bl(\reals)$, with smoothing parameter $\mu = 0.1$. As $\mu$ goes to 0, functions $f_\mu,h_\mu$ tend to a pathological behavior where all iterates are equal and for which we have exactly $\overf(\overx_N) - f_* = D_{\overh}(\overx_*,\overx_0) / N$.}\label{fig:worstcase1d}
	\end{figure}
% \end{center}